%% file: preprint-main.tex
\documentclass{article}

\usepackage[twoside,
paperwidth=210mm,
paperheight=297mm,
textheight=622pt,
textwidth=468pt,
centering,
headheight=50pt,
headsep=12pt,
footskip=18pt,
footnotesep=24pt plus 2pt minus 12pt,
columnsep=2pc]{geometry}

\usepackage[auth-sc]{authblk}

\input{latex-commands.tex}
\usepackage{amsthm}

\usepackage[capitalise]{cleveref}

\newtheorem{theorem}{Theorem}[section]
\newtheorem{lemma}[theorem]{Lemma}
\newtheorem{proposition}[theorem]{Proposition}
\newtheorem{corollary}{Corollary}[theorem]
\theoremstyle{definition}

\newtheorem{assumption}{Assumption}

\crefname{assumption}{assumption}{assumptions}
\Crefname{assumption}{Assumption}{Assumptions}

\crefname{problem}{problem}{problems}
\Crefname{problem}{Problem}{Problems}
\crefname{equation}{}{}
\Crefname{equation}{}{}

\theoremstyle{remark}

\usepackage{lineno}

\title{Homogenization rates of beam lattices to micropolar continua}
\author[1]{Eric T. Chung}
\author[1]{Kuang Huang}
\author[1]{Changqing Ye\thanks{\href{mailto:cqye@math.cuhk.edu.hk}{cqye@math.cuhk.edu.hk}}}
\affil[1]{Department of Mathematics, The Chinese University of Hong Kong, Shatin, Hong~Kong~SAR, China.}

\begin{document}
\maketitle
\begin{abstract}
\input{abstract.tex}

\end{abstract}

\input{main-text.tex}

\section*{Acknowledgments}
\input{acknowledgements.tex}

\bibliographystyle{siamplain}
\bibliography{refs}
\end{document}

%% file: latex-commands.tex
\usepackage{mathtools}

\allowdisplaybreaks

\usepackage{amssymb}
\usepackage[mathscr]{eucal}

\usepackage{mismath}

\usepackage{hyperref}

\usepackage{bm} 

\usepackage{graphicx}
\graphicspath{{figs/}}

\usepackage{tikz}
\usetikzlibrary{math}
\usetikzlibrary{calc}

\usepackage{subcaption}

\usepackage{multirow}

\usepackage{makecell}
\setcellgapes{2pt}

\usepackage{booktabs}

\usepackage{siunitx}
\usepackage{algorithm}
\usepackage{algpseudocode}

\usepackage{relsize}
\newcommand{\SSSText}[1]{{\mathup{\mathsmaller{#1}}}}

\newcommand{\Real}{\mathbb{R}}
\newcommand{\Complex}{\mathbb{C}}
\newcommand{\Integer}{\mathbb{Z}}

\DeclareMathOperator{\Div}{div}

\DeclarePairedDelimiter{\RoundBrackets}{(}{)}
\DeclarePairedDelimiter{\CurlyBrackets}{\{}{\}}
\DeclarePairedDelimiter{\SquareBrackets}{[}{]}

\newcommand{\mathdefault}[1][]{}

\usepackage[skins]{tcolorbox}
\newtcolorbox{justabox}[2][]{%
  enhanced,
  attach boxed title to top center={yshift=-3mm,yshifttext=-1mm},
  colframe=blue!75!black,
  colbacktitle=red!80!black,
  fonttitle=\bfseries,
  title=#2,#1
}

\usepackage{wasysym}

%% file: abstract.tex
As the size of a mechanical lattice with beam-modeled edges approaches zero, it undergoes homogenization into a continuum model, which exhibits unusual mechanical properties that deviate from classical Cauchy elasticity, named micropolar elasticity.
Typically, the homogenization process is qualitative in the engineering community, lacking quantitative homogenization error estimates.
In this paper, we rigorously analyze the homogenization process of a beam lattice to a continuum.
Our approach is initiated from an engineered mechanical problem defined on a triangular lattice with periodic boundary conditions.
By applying Fourier transformations, we reduce the problem to a series of equations in the frequency domain.
As the lattice size approaches zero, this yields a homogenized model in the form of a partial differential equation with periodic boundary conditions.
This process can be easily justified if the external conditions in the frequency domain are nonzero only at low-frequency modes.
However, through numerical experiments, we discover that beyond the low-frequency regime, the homogenization of the beam lattice differs from classical periodic homogenization theory due to the additional rotational degrees of freedom in the beams.
A crucial technique in our analysis is the decoupling of displacement and rotation fields, achieved through a linear algebraic manipulation known as the Schur complement.
Through dedicated analysis, we establish the coercivity of the Schur complements in both lattice and continuum models, which enables us to derive convergence rate estimates for homogenization errors.
Numerical experiments validate the optimality of the homogenization rate estimates.

%% file: main-text.tex
\section{Introduction}
Over the past few decades, the development of advanced fine additive manufacturing techniques has revolutionized the design and fabrication of metamaterials---engineered materials that exhibit properties not naturally found in conventional materials.
A significant number of mechanical metamaterials are characterized by lattice structures, where sub-dimensional components, i.e., edges, are commonly modeled as beams, rods, or thin walls \cite{Somnic2022,Jiao2023}.
By precisely manipulating the geometry and arrangement of these lattice structures, mechanical metamaterials can achieve an exceptional balance of high stiffness and low weight, which is especially advantageous in fields such as aerospace, civil engineering, and biomedical engineering \cite{Bici2018,Berger2017,Ha2020}.
In our work, we specifically focus on beam lattice metamaterials, where all edges function as mechanical beams, and all nodes serve as beam junctions \cite{Fleck2010,Meza2017}.


A standard approach to obtaining the homogenized model of a beam lattice involves first constructing a Representative Volume Element (RVE), then calculating the potential energy on the RVE, and finally determining the unknown parameters in the prescribed continuous model, commonly micropolar elasticity, by equating the potential energy density on the RVE \cite{Kumar2004,Hasanyan2016,Alavi2022}.
While this method is effective, it lacks quantitative insights into how the discrete model deviates from the homogenized model.
The concept of quantitative homogenization, introduced by the PDE analysis community, aims to measure the error between multiscale models and their homogenized counterparts.
This theory typically involves the following settings: (1) the coefficient in the multiscale PDE exhibits small-periodic oscillations; (2) both the multiscale and homogenized solutions are derived from PDEs with identical boundary conditions; and (3) the difference between the two solutions is quantified under suitable regularity assumptions on the provided data \cite{Jikov1994,Shen2018,Ye2021,Ye2023b}.
We will adopt the general methodology of quantitative homogenization and tailor it to the specific context of beam lattice models.

Quantitative homogenization, in essence, provides an estimate of the convergence rates of a lattice model to its continuum counterpart.
Therefore, understanding general convergence results is a prerequisite for achieving a quantitative outcome.
The theory of $\Gamma$-convergence, which examines the convergence of energy functionals, serves as a suitable mathematical framework for this purpose \cite{DalMaso1993,Braides2002}.
A significant contribution to this field was established by Abdoul-Anziz and Seppecher \cite{AbdoulAnziz2018a,AbdoulAnziz2018} who offered a comprehensive picture of how a high-contrast elastic medium first converges to a sub-dimensional model (i.e., a lattice) and subsequently to a continuum model incorporating second-order gradient terms.
Additionally, they addressed a common misconception---gradient effects arise from flexural interactions, emphasizing the importance of rigorous mathematical analysis over mechanical presumptions in studying the homogenization process.
Meanwhile, we also note the recent work by Griso, Khilkova, and Orlik et al.\ which highlights the intrinsic mathematical complexity of mechanical lattice homogenization \cite{Griso2021}.
Depending on the homogenization process and parameter scaling, the limiting model can be strikingly different, e.g., the high-order gradient terms emphasized in Ref.\ \cite{AbdoulAnziz2018a} diminishes in the setting described in Ref.\ \cite{Griso2021}.


The thrust of our work is mimicking the quantitative homogenization framework in PDEs.
The first challenge is deriving the homogenized model of a beam lattice.
Unlike conventional methods that rely on Taylor expansions of macro displacements and rotations at the nodes, We devise a mechanical problem on the lattice where each node is subjected to external forces and torques.
To exclude the influence of boundary conditions, we exclusively consider Periodic Boundary Conditions (PBC), and the equilibrium solution to this mechanical problem can be expressed in the Fourier domain through discrete Fourier transformations, with the corresponding equation for each Fourier mode containing the scale parameter.
Under a suitable scaling process and taking a limit, we derive the equation corresponding to the Fourier coefficient, which determines the homogenized model in the context of Fourier expansions.
Therefore, the homogenized model, represented as a PDE defined on the continuum domain with PBCs, is rigorously established.
The mathematical validity of this procedure is straightforward to verify, assuming that the Fourier modes of external forces and torques are non-zero only at low frequencies---a concept we refer to as the \emph{low-frequency principle}.
The difference between the two solutions is hence the error induced by the homogenization process.
However, on the triangular lattice, we numerically observe that the error cannot diminish to zero as the scale parameter approaches zero when only the $L^2$-type right-hand terms---the external forces and torques---are provided.
This observation reveals a significant distinction between beam lattice homogenization and classical periodic homogenization \cite{Shen2018}.

Given the differing physical interpretations of displacements and rotations, our quantitative analysis begins by decoupling these two types of degrees of freedom, which naturally leads to the rediscovery of the Schur complement corresponding to the displacement component.
Through a detailed analysis, we prove that the Schur complements in the lattice and continuum models are both coercive, akin to the Laplace operator---a property that can also be interpreted as the Korn inequality \cite{Duvaut1976}.
Leveraging this coercivity property, we establish a quantitative estimate for the homogenization error, where the convergence rates regarding the scale parameter are explicitly determined.
Our analysis implies that achieving the same level of accuracy for both displacement and rotation fields requires the regularity of external torques to be higher than that of external forces.
Furthermore, we extend our analysis to the square lattice.

We are aware of pioneering works of Martinsson and Babu\v{s}ka in \cite{Martinsson2007,Martinsson2007a} on the homogenization of mechanical truss/frame models.
Our work distinguishes itself by addressing the rotation field within the lattice, which significantly complicates the mathematical analysis.

The remainder of this paper is organized as follows.
In \cref{sec: preliminaries}, we introduce the foundational concepts of Euler--Bernoulli beam theory and Fourier transforms, which are essential for the subsequent discussions.
\Cref{sec: homogenization} presents a rational framework for understanding the homogenization process of beam lattices. Here, we focus on the low-frequency principle to derive the homogenized model and discuss numerical observations to motivate the quantitative analysis.
\Cref{sec: quantitative} is the core of the paper, where we rigorously develop the mathematical tools needed for a precise estimate of the homogenization error and the optimality of the theory is validated through numerical experiments.
Extensions to other lattice structures are discussed in \cref{sec: extensions}.
Finally, we summarize our conclusions and outlooks in \cref{sec: conclusion}.

\section{Preliminaries}\label{sec: preliminaries}
In this section, we introduce some basic concepts, notations and mathematical tools that will be used throughout the paper.

\subsection{Euler--Bernoulli beam theory}
A lattice solely is an abstraction of geometric relations of nodes and edges.
We still need to apply physical laws to pin down the specific mechanical problem that will be investigated.
In this work, every edge of the lattice is modeled as an Euler--Bernoulli beam, and the degrees of freedom of the whole system are the displacements and rotations of the nodes.
The extension to other beam models, such as Timoshenko beams, is not discussed here but can be considered in future work.
However, we emphasize that interesting phenomena emerge from the introduction of rotations, which are not presented in an oversimplified model, like the spring or truss lattice model.

We briefly review the Euler--Bernoulli beam theory.
Consider a beam with a length of $L$ that undergoes deformation solely due to boundary conditions and is not subjected to any external loading.
The governing equations are
\begin{alignat}{2}
  \frac{\di}{\di s} \RoundBrackets*{\mathtt{S} \frac{\di \bar{u}_x}{\di s}} = 0,         & \quad \forall s \in (0, L)\quad &  & \text{(axial displacement)}, \label{eq: axial}           \\
  \frac{\di^2}{\di s^2} \RoundBrackets*{\mathtt{H} \frac{\di^2 \bar{u}_y}{\di s^2}} = 0, & \quad \forall s \in (0, L)\quad &  & \text{(transverse displacement)}, \label{eq: transverse}
\end{alignat}
where $\mathtt{S}$ and $\mathtt{H}$ are the axial and centroidal bending stiffness, respectively (see Ref.\ \cite{Bauchau2009}).
Boundary conditions are applied at two ending nodes as
\[
  \left\{
  \begin{alignedat}{1}
    \bar{u}_x(0)   & = v_{x}^{\vdash},                                                         \\
    \bar{u}_y(0)   & = v_{y}^{\vdash},                                                         \\
    -\bar{u}'_y(0) & = -\theta^{\vdash} (\text{\small ``-'' is due to the outward direction}), \\
  \end{alignedat}
  \right.
  \text{ and }
  \left\{
  \begin{alignedat}{1}
    \bar{u}_x(L)  & = v_{x}^{\dashv},  \\
    \bar{u}_y(L)  & = v_{y}^{\dashv},  \\
    \bar{u}'_y(L) & = \theta^{\dashv}.
  \end{alignedat}
  \right.
\]
From \cref{eq: axial}, we can see that the axial displacement $\bar{u}_x$ is a Hermite interpolation of the boundary conditions as
\[
  \bar{u}_x(s) = v_{x}^{\vdash}\RoundBrackets*{1-t}+v_{x}^{\dashv}t \ \text{with } t \coloneqq \frac{s}{L}.
\]
Similarly, from \cref{eq: transverse}, the transverse displacement $\bar{u}_y$ can be determined by the cubic Hermite interpolation as
\[
  \bar{u}_y(s) = (2t^3-3t^2+1)v_{y}^{\vdash}+(t^3-2t^2+t)L\theta^{\vdash}+(-2t^3+3t^2)v_y^{\dashv}+(t^3-t^2)L\theta^{\dashv}
\]
with $t = s/L$ again.
\Cref{fig: EB-beam} serves as a visual illustration of the deformation of an Euler--Bernoulli beam under the imposed boundary conditions.

\begin{figure}[!ht]
  \centering
  \input{figs/EB-beam.pgf}
  \caption{An Euler--Bernoulli beam. The deformation of the beam is driven by boundary conditions, i.e., the displacements $(v_{x}^{\vdash}, v_{y}^{\vdash})$ and $(v_{x}^{\dashv}, v_{y}^{\dashv})$, and rotations $\theta^{\vdash}$ and $\theta^{\dashv}$ at two ending nodes.}\label{fig: EB-beam}
\end{figure}
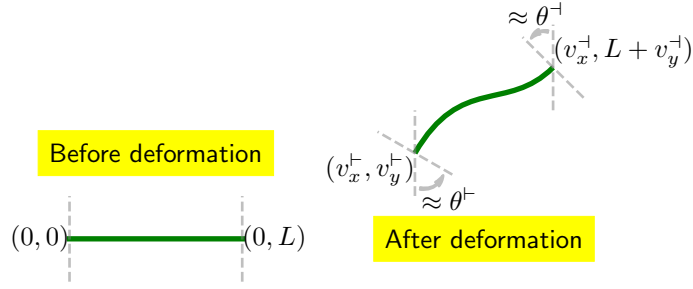

The potential energy of the beam is given by
\begin{equation}\label{eq: potential}
  \begin{split}
     & \quad \frac{1}{2}\int_{0}^{L} \mathtt{S}\RoundBrackets*{\bar{u}_x'}^2 + \mathtt{H} \RoundBrackets*{\bar{u}_y''}^2 \di s = \frac{\mathtt{S}}{2L}\RoundBrackets*{v_x^{\vdash} - v_x^{\dashv}}^2 + \frac{\mathtt{H}}{2L^3} \int_{0}^{1} \RoundBrackets*{\frac{\di^2 \bar{u}_y}{\di t^2}}^2 \di t \\
     & =\frac{\mathtt{S}}{2L}\begin{bmatrix}
                               v_x^{\vdash} & v_x^{\dashv}
                             \end{bmatrix} \cdot \mathscr{K}_{\tau}\begin{bmatrix}
                                                                     v_x^{\vdash} \\
                                                                     v_x^{\dashv}
                                                                   \end{bmatrix}+\frac{\mathtt{H}}{2L^3}\begin{bmatrix}
                                                                                                          v_y^{\vdash} & \theta^{\vdash} & v_y^{\dashv} & \theta^{\dashv}
                                                                                                        \end{bmatrix} \cdot \mathscr{K}_{\nu}\begin{bmatrix}
                                                                                                                                               v_y^{\vdash}    \\
                                                                                                                                               \theta^{\vdash} \\
                                                                                                                                               v_y^{\dashv}    \\
                                                                                                                                               \theta^{\dashv}
                                                                                                                                             \end{bmatrix},
  \end{split}
\end{equation}
where
\[
  \mathscr{K}_{\tau} \coloneqq \begin{bmatrix}
    1  & -1 \\
    -1 & 1
  \end{bmatrix}
  \text{ and }
  \mathscr{K}_{\nu} \coloneqq \begin{bmatrix}
    12  & 6L   & -12 & 6L   \\
    6L  & 4L^2 & -6L & 2L^2 \\
    -12 & -6L  & 12  & -6L  \\
    6L  & 2L^2 & -6L & 4L^2
  \end{bmatrix}
\]
are two positive semi-definite matrices.
We can also rewrite \cref{eq: potential} into a more compact form by noting the following relations:
\begin{align*}
  \begin{bmatrix}
    v_x^{\vdash} & v_x^{\dashv}
  \end{bmatrix} \cdot \mathscr{K}_{\tau}\begin{bmatrix}
                                          v_x^{\vdash} \\
                                          v_x^{\dashv}
                                        \end{bmatrix}            & =\RoundBrackets*{v_x^{\vdash}-v_x^{\dashv}}^2,                                                                                        \\
  \begin{bmatrix}
    v_y^{\vdash} & \theta^{\vdash} & v_y^{\dashv} & \theta^{\dashv}
  \end{bmatrix} \cdot \mathscr{K}_{\nu}\begin{bmatrix}
                                         v_y^{\vdash}    \\
                                         \theta^{\vdash} \\
                                         v_y^{\dashv}    \\
                                         \theta^{\dashv}
                                       \end{bmatrix} & = L^2\RoundBrackets*{\theta^{\vdash}-\theta^{\dashv}}^2 + 3L^2\RoundBrackets*{2v_y^{\vdash}/L-2v_y^{\dashv}/L+\theta^{\vdash}+\theta^{\dashv}}^2.
\end{align*}
Those relations are useful for the subsequent derivations.
Therefore, in calculating the potential energy of the beam, two directions should be considered: the axial direction and the transverse direction.
If the axial direction is denoted by $\bm{l}$, we use $\bm{l}^\perp$ to represent the transverse direction, which is generated by rotating $\bm{l}$ counterclockwise by $\pi/2$.

\subsection{Fourier transforms}
Our analysis heavily relies on Fourier transforms, in both the discrete and continuous settings.
We define $\mathscr{I}_N \coloneqq \CurlyBrackets*{(i,j)\in \Integer^2 \colon 0 \leq i,j < N}$ as a set for indexes.
Any discrete scalar/vector-valued function $f$ can be identified as a map from $\mathscr{I}_N$ to the corresponding co-domain.
In the paper, we use $f[i,j]$ to refer to the value of $f$ at the index $(i,j)$.
We also adopt a cyclic notation for the indexes, i.e., $f[i,j]=f[i\mod N, j\mod N]$.
The Discrete Fourier Transform (DFT) of $f$ yields $\hat{f}$ (or $f^\wedge$), defined as
\[
  \hat{f}[i',j']=\frac{1}{N^2}\sum_{0\leq i,j < N} f[i,j]\exp(-2\pi\i (i i'+j j')/N),
\]
where $-N/2\leq i',j' < N/2$, $\i$ is the imaginary unit, and the summation is applied entry-wise on $f$ as for the vector-valued case.
We also denote
\[
  \mathscr{F}_N \coloneqq \CurlyBrackets*{(i',j')\in \Integer^2 \colon -N/2 \leq i,j < N/2}
\]
as the frequency domain and $\mathscr{F}_N^\circ \coloneqq \mathscr{F}_N \setminus \CurlyBrackets*{(0,0)}$ the non-zero frequency domain.
The inverse DFT ($\hat{f} \mapsto f$) is given by
\[
  f[i,j]=\sum_{-N/2\leq i',j' < N/2} \hat{f}[i',j']\exp(2\pi\i (i i'+j j')/N).
\]
As for differences, indexes in the frequency domain are marked with ``$\prime$'' to distinguish them from the indexes in the spatial domain.
The inclusion of the weights $1/N^2$ and $1$ in the DFT and inverse DFT, respectively, differs from conventional definitions but is maintained for consistency with Fourier expansions, as will be shown shortly.
Parseval's identity states that
\[
  \sum_{(i',j')\in \mathscr{F}_N} \hat{g}^*[i',j'] \hat{f}[i',j'] = \frac{1}{N^2} \sum_{(i,j)\in \mathscr{I}_N} g^*[i,j] f[i,j],
\]
where the asterisk superscripted $(\cdot)^*$ denotes the complex conjugate.

We denote $D \coloneqq (0,1)^2 \subset \Real^2$ as the unit square domain.
With a slight abuse of notation, for a scalar-valued function $f\in L^2(D)$, we can take the Fourier expansion as
\[
  f(\alpha,\beta) = \sum_{(i',j')\in \Integer^2} \hat{f}[i',j']\exp(2\pi\i (i' \alpha+j' \beta))
\]
for a.e.\ $(\alpha,\beta)\in D$, where
\[
  \hat{f}[i',j'] = \int_{D} f(\alpha,\beta)\exp(-2\pi\i (i' \alpha+j' \beta))\di \alpha \di \beta.
\]
Thanks to the orthogonality of the trigonometric bases, we have
\[
  \norm{f}_{0}^2=\sum_{(i',j')\in \Integer^2} \abs{\hat{f}[i',j']}^2,
\]
where $\norm{\cdot}_{0}$ is a shorthand notation of $L^2(D)$ norm.
We also define a trigonometric polynomial subspace
\[
  \mathscr{T}_N\coloneqq \CurlyBrackets*{\sum_{(i',j')\in \mathscr{F}_N}c[i',j']\exp(2\pi \i (i'\alpha+j'\beta))\colon c[i',j'] \in \Complex} \subset L^2(D).
\]
Then for $g\in \mathscr{T}_N$, its evaluations at grid points $\CurlyBrackets*{(i/N,j/N)}$ determine a discrete function $h[i,j]$.
We can easily check that $\hat{g}[i',j']=\hat{h}[i',j']$ for $(i',j')\in \mathscr{F}_N$, which justifies the choice of weights in the DFT and inverse DFT.
Moreover, for $f\in L^2(D)$, we can introduce $H^s(D)$ semi-norms for higher regularity as
\[
  \abs{f}_{s}^2 \coloneqq \sum_{(i',j')\in \Integer^2} \RoundBrackets*{\abs{i'}^{2s}+\abs{j'}^{2s}} \abs{\hat{f}[i',j']}^2,
\]
where $s>0$.
We denote the function space $H^s(D)$ for functions with finite $L^2(D)$ and $H^s(D)$, which essentially is the conventional Sobolev space in the Fourier series setting.
For discrete functions, the $H^s$ semi-norm is defined similarly, while the Fourier coefficients are replaced by the inverse DFT and the summation is over $(i',j')\in \mathscr{F}_N$.
Hence, if we quantify the difference between a discrete function $f_{\SSSText{D}}$ defined on $\mathscr{I}_N$ and a continuous $f_{\SSSText{C}}\in \mathscr{T}_N$ in the $L^2$ norm or $H^s$ semi-norm, it is reasonable to calculate
\begin{align*}
              & \RoundBrackets*{\sum_{(i',j')\in \mathscr{F}_N} \abs{\hat{f}_{\SSSText{D}}[i',j']-\hat{f}_{\SSSText{C}}[i',j']}^2}^{1/2}                                              \\
  \text{ or } & \RoundBrackets*{\sum_{(i',j')\in \mathscr{F}_N} \RoundBrackets*{\abs{i'}^{2s}+\abs{j'}^{2s}} \abs{\hat{f}_{\SSSText{D}}[i',j']-\hat{f}_{\SSSText{C}}[i',j']}^2}^{1/2}
\end{align*}
as a proper measure.

\section{Homogenization with the low-frequency principle}\label{sec: homogenization}
To analyze the homogenization process of beam lattices, we establish a general procedure to rigorously formulate the final mathematical statement.
The procedure involves the following steps:
\begin{enumerate}
  \item Identify the periodic cell and the translation vectors $\bm{t}_x$ and $\bm{t}_y$, ensuring that the periodic cell can tessellate the plane, with $\abs{\bm{t}_x}=\abs{\bm{t}_y}=1$. Translate the periodic cell by $N\times N$ copies along $\bm{t}_x$ and $\bm{t}_y$, scale it by $\epsilon\coloneqq 1/N$, and construct the domain $\Omega_\epsilon$.
  \item Within a periodic cell, select the nodes to place the degrees of freedom and edges for calculating the total energy.
  \item Introduce fictitious forces and torques for each node and derive the balance equations for all nodes.
  \item Take the limit as $\epsilon$ approaches zero to derive the continuous (homogenized) model for the balance equations with the \emph{low-frequency principle}, resulting in a system of PDEs with PBCs.
  \item Provide a quantitative estimate in terms of $\epsilon$ for the difference between the discrete and continuous models.
\end{enumerate}

We remark that the domain $\Omega_\epsilon$ may exhibit rough boundaries.
However, as $\epsilon$ approaches to zero, $\Omega_\epsilon$ will ``converge'' to a domain $\Omega_0$.
The homogenized model will be established on $\Omega_0$.
In this section, we focus on the triangular lattice, as illustrated in \cref{fig: triangular lattice}.
The periodic cell is depicted in the left part of \cref{fig: triangular lattice}, with a node marked by a red cross.
Only one node is drawn in the periodic cell because the other nodes can be generated by translating the periodic cell.
The three beams in the periodic cell are colored differently and can also be referred to by their directions: $\bm{l}_1=[1/2, -\sqrt{3}/2]^\intercal$, $\bm{l}_2=[1, 0]^\intercal$, and $\bm{l}_3=[1/2, \sqrt{3}/2]^\intercal$.
The periodic cell is translated by $10\times 10$ copies along $\bm{t}_x=[1/2, -\sqrt{3}/2]^\intercal$ and $\bm{t}_y=[1/2, \sqrt{3}/2]^\intercal$, which is shown in the middle part of \cref{fig: triangular lattice}.
We can label the node in the periodic cell by $(i,j)$, where $i$ and $j$ belong to $\CurlyBrackets{0, 1,\dots, N}$.
We specifically define the coordinates of the node $(i,j)$ by $i\epsilon \bm{t}_x+ j\epsilon \bm{t}_y$, such that the $N\times N$ copies of the periodic cell occupy the domain $\CurlyBrackets*{(x, y)=v_x \bm{t}_x+v_y\bm{t}_y\colon 0< v_x, v_y < 1}$, which is exactly the domain $\Omega_0$ mentioned above.

\begin{figure}[!ht]
  \centering
  \resizebox{\textwidth}{!}{\input{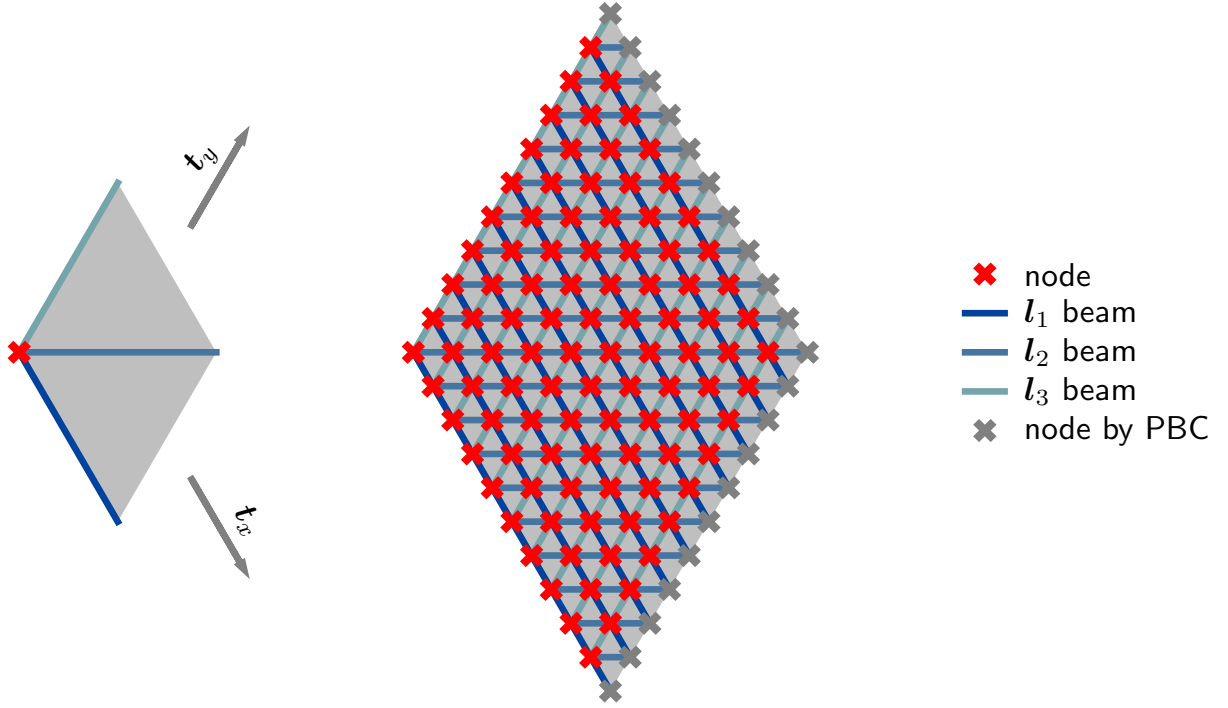}}
  \caption{The demonstration of the triangular lattice.
    (Left) A periodic cell with a node marked, and three beams colored differently and the directions of the translation vectors $\bm{t}_x$ and $\bm{t}_y$.
    (Middle) The domain $\Omega_\epsilon$ constructed by translating the periodic cell $10\times 10$ copies along $\bm{t}_x$ and $\bm{t}_y$.
    (Right) Some notations.}
  \label{fig: triangular lattice}
\end{figure}

\subsection{The discrete model}\label{subsec: discrete}
The displacement of the node $(i,j)$ is denoted by $\bm{u}_{\SSSText{D}}[i,j]=[u^x_{\SSSText{D}}[i,j], u^y_{\SSSText{D}}[i,j]]^\intercal$, and the rotation of the node is denoted by $\theta_{\SSSText{D}}[i,j]$.
We assume a PBC for the displacement and rotation fields:
\begin{alignat}{3}
   & \bm{u}_{\SSSText{D}}[0,j]=\bm{u}_{\SSSText{D}}[N,j], & \quad & \theta_{\SSSText{D}}[0,j]=\theta_{\SSSText{D}}[N,j], & \quad & \forall j\in \CurlyBrackets{0, 1,\dots, N}; \nonumber        \\
   & \bm{u}_{\SSSText{D}}[i,0]=\bm{u}_{\SSSText{D}}[i,N], & \quad & \theta_{\SSSText{D}}[i,0]=\theta_{\SSSText{D}}[i,N], & \quad & \forall i\in \CurlyBrackets{0, 1,\dots, N}.  \label{eq: PBC}
\end{alignat}
Thus, the total degrees of freedom of this system can be treated as $\CurlyBrackets{\bm{u}_{\SSSText{D}}[i,j], \theta_{\SSSText{D}}[i,j]}$ for $(i,j)\in \mathscr{I}_N$.
In the middle part of \cref{fig: triangular lattice}, the boundary nodes determined by the PBCs are colored differently for distinction.

According to the setting, all beams have the same length, which is exactly $\epsilon$.
The energy of the system is given by $\mathcal{L}=\mathcal{L}_1+\mathcal{L}_2+\mathcal{L}_3$, where $\mathcal{L}_\diamond$ is the energy associated with the beam in the direction $\bm{l}_\diamond$ with $\diamond\in\CurlyBrackets{1,2,3}$.
Precisely, we have
\begin{alignat}{1}
   & \quad \mathcal{L}_1(\bm{u}_{\SSSText{D}}, \theta_{\SSSText{D}})                                                                                                                                                                                                    \nonumber                                                      \\
   & = \frac{1}{2} \sum_{(i,j)\in \mathscr{I}_N} \Big\{\mathtt{S}\RoundBrackets*{(\bm{u}_{\SSSText{D}}[i,j]-\bm{u}_{\SSSText{D}}[i+1,j])\cdot \bm{l}_1}^2/ \epsilon + \mathtt{H} \RoundBrackets*{\theta_{\SSSText{D}}[i,j]- \theta_{\SSSText{D}}[i+1,j]}^2/\epsilon     \nonumber                                                      \\
   & \qquad + 3\mathtt{H}\RoundBrackets*{2(\bm{u}_{\SSSText{D}}[i,j]-\bm{u}_{\SSSText{D}}[i+1,j])\cdot \bm{l}_1^\perp/\epsilon+ \theta_{\SSSText{D}}[i,j]+ \theta_{\SSSText{D}}[i+1,j]}^2/\epsilon \Big\};                                                                                                         \label{eq: trl en1} \\
   & \quad \mathcal{L}_2(\bm{u}_{\SSSText{D}}, \theta_{\SSSText{D}})                                                                                                                                                                                                   \nonumber                                                       \\
   & = \frac{1}{2} \sum_{(i,j)\in \mathscr{I}_N} \Big\{\mathtt{S}\RoundBrackets*{(\bm{u}_{\SSSText{D}}[i,j]-\bm{u}_{\SSSText{D}}[i+1,j+1])\cdot \bm{l}_2}^2/\epsilon + \mathtt{H} \RoundBrackets*{\theta_{\SSSText{D}}[i,j]- \theta_{\SSSText{D}}[i+1,j+1]}^2/\epsilon \nonumber                                                       \\
   & \qquad + 3\mathtt{H}\RoundBrackets*{2(\bm{u}_{\SSSText{D}}[i,j]-\bm{u}_{\SSSText{D}}[i+1,j+1])\cdot \bm{l}_2^\perp/\epsilon+ \theta_{\SSSText{D}}[i,j]+ \theta_{\SSSText{D}}[i+1,j+1]}^2/\epsilon \Big\};                                                                                                     \label{eq: trl en2} \\
   & \quad \mathcal{L}_3(\bm{u}_{\SSSText{D}}, \theta_{\SSSText{D}})                                                                                                                                                                                                   \nonumber                                                       \\
   & = \frac{1}{2} \sum_{(i,j)\in \mathscr{I}_N} \Big\{\mathtt{S}\RoundBrackets*{(\bm{u}_{\SSSText{D}}[i,j]-\bm{u}_{\SSSText{D}}[i,j+1])\cdot \bm{l}_3}^2/\epsilon + \mathtt{H} \RoundBrackets*{\theta_{\SSSText{D}}[i,j]- \theta_{\SSSText{D}}[i,j+1]}^2/\epsilon     \nonumber                                                       \\
   & \qquad + 3\mathtt{H}\RoundBrackets*{2(\bm{u}_{\SSSText{D}}[i,j]-\bm{u}_{\SSSText{D}}[i,j+1])\cdot \bm{l}_3^\perp/\epsilon+ \theta_{\SSSText{D}}[i,j]+ \theta_{\SSSText{D}}[i,j+1]}^2/\epsilon \Big\}. \label{eq: trl en3}
\end{alignat}
We emphasize that due to the PBC, energy terms corresponding to several boundary beams are not included in the above expressions (refer to the middle part of \cref{fig: triangular lattice}).

According to the Euler--Bernoulli beam theory \cite{Bauchau2009}, we have the following relations:
\[
  \mathtt{S}=\mathbb{E} A \text{ and } \mathtt{H}= \mathbb{E} I,
\]
where $\mathbb{E}$ is the Young's modulus, $A$ is the cross-sectional area, and $I$ is the second moment of the cross-section.
Here, $\mathbb{E}$ has a unit of $[\text{Pressure}]$, $A$ has a unit of $[\text{Length}]^2$, and $I$ has a unit of $[\text{Length}]^4$.
The values of $A$ and $I$ are determined entirely by the geometry of the beam.
To guarantee that the Euler--Bernoulli beam model remain effective as the beam length approaches zero, $\mathtt{S}$ and $\mathtt{H}$ must be scaled appropriately, which leads to the following assumption.
\begin{assumption}
  Let $l$ be the length of the beam, then there exist parameters $\rho^\star$ and $\gamma$ such that
  \[
    \mathtt{S}=\gamma \rho^\star l \text{ and } \mathtt{H}= \gamma l^3,
  \]
  where $\rho^\star$ is a dimensionless parameter independent of $l$.
\end{assumption}
The parameter $\gamma$ cannot be dimensionless, and we will later introduce scaling factors on fictitious forces and torques to cancel $\gamma$.

An important property of the DFT is that it can diagonalize shift operations.
For instance, for a discrete function $f$ defined on $\mathscr{I}_N$, represented as the map $(i, j) \mapsto f[i,j]$, it holds that
\begin{align*}
  ((i,j) \mapsto f[i+1,j])^{\wedge}   & = (i',j')\mapsto \exp(2\pi\i i'\epsilon) \hat{f}[i',j'],     \\
  ((i,j) \mapsto f[i,j+1])^{\wedge}   & = (i',j')\mapsto \exp(2\pi \i j'\epsilon)\hat{f}[i',j'],     \\
  ((i,j) \mapsto f[i+1,j+1])^{\wedge} & = (i',j')\mapsto \exp(2\pi\i (i'+j')\epsilon)\hat{f}[i',j'],
\end{align*}
where now $1/N$ is replaced by $\epsilon$.
Utilizing Parseval's identity and the shift property, the potential energy in \cref{eq: trl en1,eq: trl en2,eq: trl en3} can be reformulated in the Fourier space as
\[
  \mathcal{L}(\bm{u}_{\SSSText{D}}, \theta_{\SSSText{D}}) = \frac{\gamma}{2} \sum_{(i',j') \in \mathscr{F}_N}
  \begin{bmatrix}
    \hat{\bm{u}}_{\SSSText{D}}^*[i',j'] \\ \hat{\theta}^*_{\SSSText{D}}[i',j']
  \end{bmatrix} \cdot
  \underbrace{
    \begin{bmatrix}
      A_{\SSSText{D}}[i',j']                 & \bm{b}_{\SSSText{D}}[i',j'] \\
      -\bm{b}_{\SSSText{D}}^\intercal[i',j'] & c_{\SSSText{D}}[i',j']
    \end{bmatrix}
  }_{\coloneqq S_{\SSSText{D}}[i',j']}
  \begin{bmatrix}
    \hat{\bm{u}}_{\SSSText{D}}[i',j'] \\ \hat{\theta}_{\SSSText{D}}[i',j']
  \end{bmatrix},
\]
where the $2$-by-$2$ matrix $A_{\SSSText{D}}[i',j']$, the vector $\bm{b}_{\SSSText{D}}[i',j']$, and the scalar $c_{\SSSText{D}}[i',j']$ are split into three parts corresponding to the three beams:
\begin{multline*}
  A_{\SSSText{D}}=A_{1,\SSSText{D}}+A_{2,\SSSText{D}}+A_{3,\SSSText{D}} \text{ with } A_{1,\SSSText{D}}=4(\rho^\star \bm{l}_1 \otimes \bm{l}_1 + 12 \bm{l}_1^\perp \otimes \bm{l}_1^\perp)\frac{\sin^2(\pi i'\epsilon)}{\epsilon^2},\\
  A_{2,\SSSText{D}}=4(\rho^\star \bm{l}_2 \otimes \bm{l}_2 + 12 \bm{l}_2^\perp \otimes \bm{l}_2^\perp)\frac{\sin^2(\pi (i'+j')\epsilon)}{\epsilon^2},        \\
  \text{and } A_{3,\SSSText{D}}=4(\rho^\star \bm{l}_3 \otimes \bm{l}_3 + 12 \bm{l}_3^\perp \otimes \bm{l}_3^\perp)\frac{\sin^2(\pi j'\epsilon)}{\epsilon^2};
\end{multline*}
\begin{multline*}
  \bm{b}_{\SSSText{D}}=\bm{b}_{1,\SSSText{D}}+\bm{b}_{2,\SSSText{D}}+\bm{b}_{3,\SSSText{D}} \text{ with } \bm{b}_{1,\SSSText{D}}=12\i \bm{l}_1^\perp \frac{\sin(2\pi i'\epsilon)}{\epsilon},\\
  \bm{b}_{2,\SSSText{D}}=12\i \bm{l}_2^\perp\frac{\sin(2\pi (i'+j')\epsilon)}{\epsilon},\ \text{and } \bm{b}_{3,\SSSText{D}}=12\i \bm{l}_3^\perp \frac{\sin(2\pi j'\epsilon)}{\epsilon};
\end{multline*}
\begin{multline*}
  c_{\SSSText{D}}=c_{1,\SSSText{D}}+c_{2,\SSSText{D}}+c_{3,\SSSText{D}} \text{ with } c_{1,\SSSText{D}}=12-8\sin^2(\pi i'\epsilon),\\
  c_{2,\SSSText{D}}=12-8\sin^2(\pi (i'+j')\epsilon),\ \text{and } c_{3,\SSSText{D}}=12-8\sin^2(\pi j'\epsilon).
\end{multline*}
In the above expressions, we drop $\epsilon$ also $[i',j']$ after $A_{\SSSText{D}}$, $\bm{b}_{\SSSText{D}}$, $c_{\SSSText{D}}$ and similar terms for simplicity.
We may follow this convention if no confusion arises.
Note that for all $(i',j') \in \mathscr{F}_N$, the $3$-by-$3$ matrix $S_{\SSSText{D}}[i',j']$ defined above is a Hermitian matrix.
It is easy to see that if $i'=j'=0$, the matrix $S_{\SSSText{D}}[0,0]$ is singular.

For the later reference, we also call the objects $\CurlyBrackets{A_{1,\SSSText{D}},\bm{b}_{1,\SSSText{D}},c_{\SSSText{D}}}$ as the ``$x$-mode'', $\CurlyBrackets{A_{2,\SSSText{D}},\bm{b}_{2,\SSSText{D}},c_{\SSSText{D}}}$ as the ``$xy$-mode'', and $\CurlyBrackets{A_{3,\SSSText{D}},\bm{b}_{3,\SSSText{D}},c_{\SSSText{D}}}$ as the ``$y$-mode''.


In the quantitative homogenization setting, we usually fix the right-hand side of the balance equation during the homogenization process.
Therefore, we introduce the following assumption, where a scaling factor $\gamma$ is included to simplify the balance equation.
\begin{assumption}
  Each node $(i,j) \in \mathscr{I}_N$ is subjected to a force $\epsilon^2\gamma \bm{f}[i,j]$ and a torque $\epsilon^2\gamma \tau[i,j]$.
\end{assumption}
The Lagrangian for the system is given by
\begin{equation}\label{eq: PBC Lag}
  \mathcal{L}(\bm{u}_{\SSSText{D}}, \theta_{\SSSText{D}})-\epsilon^2\gamma \sum_{(i,j)\in \mathscr{I}_N} \Big\{\bm{f}[i,j]\cdot \bm{u}_{\SSSText{D}}[i,j] + \tau[i,j]\theta_{\SSSText{D}}[i,j]\Big\}.
\end{equation}
Therefore, in the Fourier space, the balance equation for the $(i',j')$ mode can be written as
\begin{equation}\label{eq: D balance}
  S_{\SSSText{D}}[i',j']
  \begin{bmatrix}
    \hat{\bm{u}}_{\SSSText{D}}[i',j'] \\ \hat{\theta}_{\SSSText{D}}[i',j']
  \end{bmatrix}=
  \begin{bmatrix}
    \hat{\bm{f}}[i',j'] \\ \hat{\tau}[i',j']
  \end{bmatrix},
\end{equation}
where the scaling factor $\epsilon^2$ is canceled out due to Parseval's identity.
Because $S_{\SSSText{D}}[0,0]$ is singular, for the $(0, 0)$ mode, we require that
\[
  \begin{bmatrix}
    \hat{\bm{f}}[0,0] \\ \hat{\tau}[0,0]
  \end{bmatrix} \in \im \RoundBrackets*{S_{\SSSText{D}}[0,0]} =\im \RoundBrackets*{\begin{bmatrix}
      0 & 0 & 0  \\
      0 & 0 & 0  \\
      0 & 0 & 36
    \end{bmatrix}
  }
\]
to ensure that the balance equation is solvable.
Therefore, unique $\hat{u}_{\SSSText{D}}[0,0]$ and $\hat{\theta}_{\SSSText{D}}[0,0]$ can be determined in the subspace $(\ker S_{\SSSText{D}}[0,0])^\perp$.
The singularity of $S_{\SSSText{D}}[0,0]$ arises due to that a global displacement does not change the total potential energy of the beam lattice.
By taking $\epsilon$ to zero, the correct homogenized model should be derived from the above balance equation, while several arguments are needed to justify this process.

\subsection{The homogenization model}
Recall that the continuous domain $\Omega_0$ can be explicitly defined as
\[
  \Omega_0 \coloneqq  \CurlyBrackets*{\alpha\bm{t}_x + \beta\bm{t}_y\colon (\alpha, \beta)\in D} \subset \Real^2.
\]
Thus, a function defined on $D$ can be viewed as a function defined on $\Omega_0$, through the mapping $(\alpha, \beta)\mapsto \alpha\bm{t}_x + \beta\bm{t}_y$.
In the following discussion, the input of a continuous function is $(\alpha, \beta)$, and the derivatives are denoted by $\partial_\alpha$ and $\partial_\beta$, which can be pulled back to the derivatives w.r.t.\ Euclidean coordinates (i.e., $\partial_x$ and $\partial_y$) using the chain rule.
Drawing an analogy to the shift property of the DFT for Fourier expansions, we have
\[
  (\partial_\alpha f)^\wedge[i',j']=2\pi \i i' \hat{f}[i',j'], \quad (\partial_\beta f)^\wedge[i',j']=2\pi \i j' \hat{f}[i',j'].
\]
We also interpolate the fictitious point-wise forces and torques to continuous function in $\mathscr{T}_N$ as follows:
\begin{alignat*}{1}
  \bm{f}_{\SSSText{C}}(\alpha, \beta) & =\sum_{(i',j')\in \mathscr{T}_N} \hat{\bm{f}}[i',j']\exp(2\pi\i (i' \alpha+j' \beta)), \\
  \tau_{\SSSText{C}}(\alpha, \beta)   & =\sum_{(i',j')\in  \mathscr{T}_N} \hat{\tau}[i',j']\exp(2\pi\i (i' \alpha+j' \beta)).
\end{alignat*}
The continuous functions $\bm{f}_{\SSSText{C}}$ and $\tau_{\SSSText{C}}$ shall enter the homogenized PDE model as the right-hand terms.

The homogenized model must effectively capture the low-frequency modes, a concept we refer to as the low-frequency principle:
\begin{justabox}{Low-frequency principle}
  \begin{equation}\label{eq: low freq}
    \begin{multlined}
      \text{There exists a constant } M \text{ independent of } \epsilon \text{ such that} \\
      \hat{\bm{f}}[i',j']\neq \bm{0} \text{ and } \hat{\tau}[i',j']\neq 0 \text{ only for } \abs{i'}+\abs{j'}\leq M.
    \end{multlined}
  \end{equation}
\end{justabox}
In this regime, there are always only a finite number of equations corresponding to $\abs{i'}+\abs{j'}\leq M$ in \cref{eq: D balance} that need to be solved.
We can safely take the limit of $S_{\SSSText{D}}[i',j']$ as $\epsilon$ approaches zero for a fix index $(i',j')$, yielding:
\[
  S_{\SSSText{D}}[i',j'] \xrightarrow{\epsilon\to 0} S_{\SSSText{C}}[i',j']=\begin{bmatrix}
    A_{\SSSText{C}}[i',j']                 & \bm{b}_{\SSSText{C}}[i',j'] \\
    -\bm{b}_{\SSSText{C}}^\intercal[i',j'] & c_{\SSSText{C}}[i',j']
  \end{bmatrix},
\]
where $A_{\SSSText{C}}$, $\bm{b}_{\SSSText{C}}$ and $c_{\SSSText{C}}$ accordingly take the form as a summation of three parts:
\begin{multline*}
  A_{\SSSText{C}}=A_{1,\SSSText{C}}+A_{2,\SSSText{C}}+A_{3,\SSSText{C}} \text{ with } A_{1,\SSSText{C}}=4(\rho^\star \bm{l}_1 \otimes \bm{l}_1 + 12 \bm{l}_1^\perp \otimes \bm{l}_1^\perp)\pi^2(i')^2,\\
  A_{2,\SSSText{C}}=4(\rho^\star \bm{l}_2 \otimes \bm{l}_2 + 12 \bm{l}_2^\perp \otimes \bm{l}_2^\perp)\pi^2(i'+j')^2, \\
  \text{ and } A_{3,\SSSText{C}}=4(\rho^\star \bm{l}_3 \otimes \bm{l}_3 + 12 \bm{l}_3^\perp \otimes \bm{l}_3^\perp)\pi^2(j')^2;
\end{multline*}
\begin{multline*}
  \bm{b}_{\SSSText{C}}=\bm{b}_{1,\SSSText{C}}+\bm{b}_{2,\SSSText{C}}+\bm{b}_{3,\SSSText{C}} \text{ with } \bm{b}_{1,\SSSText{C}}=24\i \bm{l}_1^\perp \pi i',\\
  \bm{b}_{2,\SSSText{C}}=24\i \bm{l}_2^\perp \pi (i'+j'), \text{ and } \bm{b}_{3,\SSSText{C}}=24\i \bm{l}_3^\perp \pi j';
\end{multline*}
\[
  c_{\SSSText{C}}=c_{1,\SSSText{C}}+c_{2,\SSSText{C}}+c_{3,\SSSText{C}} \text{ with } c_{1,\SSSText{C}}=c_{2,\SSSText{C}}=c_{3,\SSSText{C}}=12.
\]
Therefore, in the Fourier space, the balance equation for the $(i',j')$ mode for the homogenized model can be written as
\begin{equation}\label{eq: C balance}
  S_{\SSSText{C}}[i',j']
  \begin{bmatrix}
    \hat{\bm{u}}_{\SSSText{C}}[i',j'] \\ \hat{\theta}_{\SSSText{C}}[i',j']
  \end{bmatrix}=
  \begin{bmatrix}
    \hat{\bm{f}}[i',j'] \\ \hat{\tau}[i',j']
  \end{bmatrix}.
\end{equation}

We can also interpret \cref{eq: C balance} as a PDE model.
Define the \emph{differential operator} $\mathcal{S}_{\SSSText{H}}$ as
\[
  \mathcal{S}_{\SSSText{H}}=\begin{bmatrix}
    A_{\SSSText{H}}                 & \bm{b}_{\SSSText{H}} \\
    -\bm{b}_{\SSSText{H}}^\intercal & c_{\SSSText{H}}
  \end{bmatrix},
\]
where
\begin{multline*}
  A_{\SSSText{H}}=A_{1,\SSSText{H}}+A_{2,\SSSText{H}}+A_{3,\SSSText{H}} \text{ with } A_{1,\SSSText{H}}=-(\rho^\star \bm{l}_1 \otimes \bm{l}_1 + 12 \bm{l}_1^\perp \otimes \bm{l}_1^\perp)\partial_{\alpha\alpha}, \\
  A_{2,\SSSText{H}}=-(\rho^\star \bm{l}_2 \otimes \bm{l}_2 + 12 \bm{l}_2^\perp \otimes \bm{l}_2^\perp)(\partial_{\alpha\alpha}+2\partial_{\alpha\beta}+\partial_{\beta\beta}),\\
  \text{ and } A_{3,\SSSText{H}}=-(\rho^\star \bm{l}_3 \otimes \bm{l}_3 + 12 \bm{l}_3^\perp \otimes \bm{l}_3^\perp)\partial_{\beta\beta};
\end{multline*}
\begin{multline*}
  \bm{b}_{\SSSText{H}}=\bm{b}_{1,\SSSText{H}}+\bm{b}_{2,\SSSText{H}}+\bm{b}_{3,\SSSText{H}} \text{ with } \bm{b}_{1,\SSSText{H}}=12\bm{l}_1^\perp \partial_\alpha, \\
  \bm{b}_{2,\SSSText{H}}=12 \bm{l}_2^\perp (\partial_\alpha+\partial_\beta), \text{ and } \bm{b}_{3,\SSSText{H}}=12\bm{l}_3^\perp \partial_\beta;
\end{multline*}
\[
  c_{\SSSText{H}}=c_{1,\SSSText{H}}+c_{2,\SSSText{H}}+c_{3,\SSSText{H}} \text{ with } c_{1,\SSSText{H}} = c_{2,\SSSText{H}}=12 = c_{3,\SSSText{H}}=12.
\]
We can establish a PDE model defined on $D$ as:
\begin{justabox}{Continuum PDE model}
  \begin{equation}\label{eq: PDE model}
    \text{Find periodic functions } \bm{u}_{\SSSText{C}} \text{ and } \theta_{\SSSText{C}} \text{ such that } \\
    \mathcal{S}_{\SSSText{H}}\begin{bmatrix}
      \bm{u}_{\SSSText{C}} \\ \theta_{\SSSText{C}}
    \end{bmatrix}=
    \begin{bmatrix}
      \bm{f}_{\SSSText{C}} \\ \tau_{\SSSText{C}}
    \end{bmatrix} \text{ in } D.
  \end{equation}
\end{justabox}
Observe that the differential operator $\mathcal{S}_{\SSSText{H}}$ is now defined on the bulk domain $D$.
Consequently, if the boundary conditions for the discrete model are modified, the boundary conditions for the homogenized model will change accordingly, while we can postulate that the PDE part remains unchanged.
Meanwhile, it can be transformed back to the Euclidean coordinates by the chain rule, resulting in a PDE model defined on $\Omega_0$.

Now, the question arises whether the homogenization model is valid beyond the low-frequency principle.

\subsection{An observation from numerical computations}
In periodic homogenization, we use a periodic function $\kappa$ to represent the coefficient within a periodic cell and a dilation factor $\epsilon$ to scale the cell (i.e., $\kappa(x/\epsilon)$), thereby representing the coefficient throughout the entire domain $\Omega$.
These coefficients are employed to formulate PDE models and for simplicity, we only consider a scalar-type PDE problem rather than a system of PDEs.
There exists a homogenized coefficient $\kappa_0$ that defines a homogenized problem, whose solution describes the macroscopic behavior as $\epsilon$ approaches zero.
A classic result in periodic homogenization provides a quantitative estimate for the difference between the multiscale and homogenized solutions with the same right-hand side:
\begin{theorem}[Ref.\ \cite{Jikov1994,Bensoussan2011,Shen2018}]\label{thm: periodic homogenization}
  Let $u_\epsilon$ be the solution to the multiscale PDE problem:
  \[
    \text{find } u_\epsilon \in H^1_0(\Omega) \text{ such that } -\Div(\kappa(x/\epsilon) \nabla u_\epsilon) = f.
  \]
  Let $u_0$ be the solution to the homogenized PDE problem:
  \[
    \text{find } u_0 \in H^1_0(\Omega) \text{ such that } -\Div(\kappa_0 \nabla u_0) = f.
  \]
  Under the regularity assumptions for the domain $\Omega$ and the coefficient $\kappa$, it holds that
  \[
    \norm{u_\epsilon-u_0}_{L^2(\Omega)} \leq C \epsilon \norm{f}_{L^2(\Omega)},
  \]
  where $C$ is a positive constant independent of $\epsilon$ and $f$.
\end{theorem}
Thus, one might anticipate a result analogous to \cref{thm: periodic homogenization} for beam lattice homogenization as
\begin{equation} \label{eq: l2 convergence wrong}
  \norm{\bm{u}_{\SSSText{D}}-\bm{u}_{\SSSText{C}}}_{0}+\norm{\theta_{\SSSText{D}}-\theta_{\SSSText{C}}}_{0} \leq C \epsilon^r\RoundBrackets*{\norm{\bm{f}}_{0}+\norm{\tau}_{0}},
\end{equation}
where $C$ and $r$ are positive constants independent of $\epsilon$ and $\bm{f}$ and $\tau$.
However, this is not the case.
Specifically, for given right-hand sides $\bm{f}$ and $\tau$ in the $L^2$ sense, the convergence in the sense of \cref{eq: l2 convergence wrong} to the homogenized model implies that
\begin{equation} \label{eq: l2 convergence}
  \max_{(i',j') \in \mathscr{F}_N^\circ} \mathbf{diff}(i',j',\epsilon) \xrightarrow{\epsilon\to 0} 0,
\end{equation}
where $\mathbf{diff}(i',j',\epsilon) \coloneqq \norm{S_{\SSSText{D}}^{-1}[i',j'] - S_{\SSSText{C}}^{-1}[i',j']}$.
It is worth noting that the maximization is taken over $\mathscr{F}_N^\circ$ rather than $\mathscr{F}_N$ due to the singularity of $S_{\SSSText{D}}[0,0]$ and $S_{\SSSText{C}}[0,0]$.
The statement \cref{eq: l2 convergence} can be tested through numerical computations, with the results presented in \cref{fig: errors}.
We select $N$ ranging from $2^2$ to $2^7$, so $\epsilon$ decreases from $1/4$ to $1/128$. As a larger $\rho^\star$ implies that tension is more dominant than bending, we vary $\rho^\star$ to study its effect on the convergence.
The left plot in \cref{fig: errors} displays the \emph{maximal} $\mathbf{diff}(i',j',\epsilon)$ for $(i',j')\in \mathscr{F}_N^\circ$, while the right plot shows the \emph{maximal} $\mathbf{diff}(i',j',\epsilon)$ for $(i',j')\in \mathscr{F}_N^\circ$ with $\abs{i'}+\abs{j'}\leq 10$.
We can simply observe from the left plot that there is no convergence for all Fourier modes as $\epsilon$ approaches zero.
For certain $\rho^\star$, the difference can even increase as $\epsilon$ decreases.
On the contrary, the right plot shows that if $\epsilon$ is sufficiently small (e.g., $\epsilon \leq 2^{-5}$), the maximal difference decays.
This indicates that the convergence is achieved uniformly for low-frequency modes, consistent with the anticipation from the low-frequency principle \eqref{eq: low freq}.

\begin{figure}[!ht]
  \centering
  \resizebox{\textwidth}{!}{\input{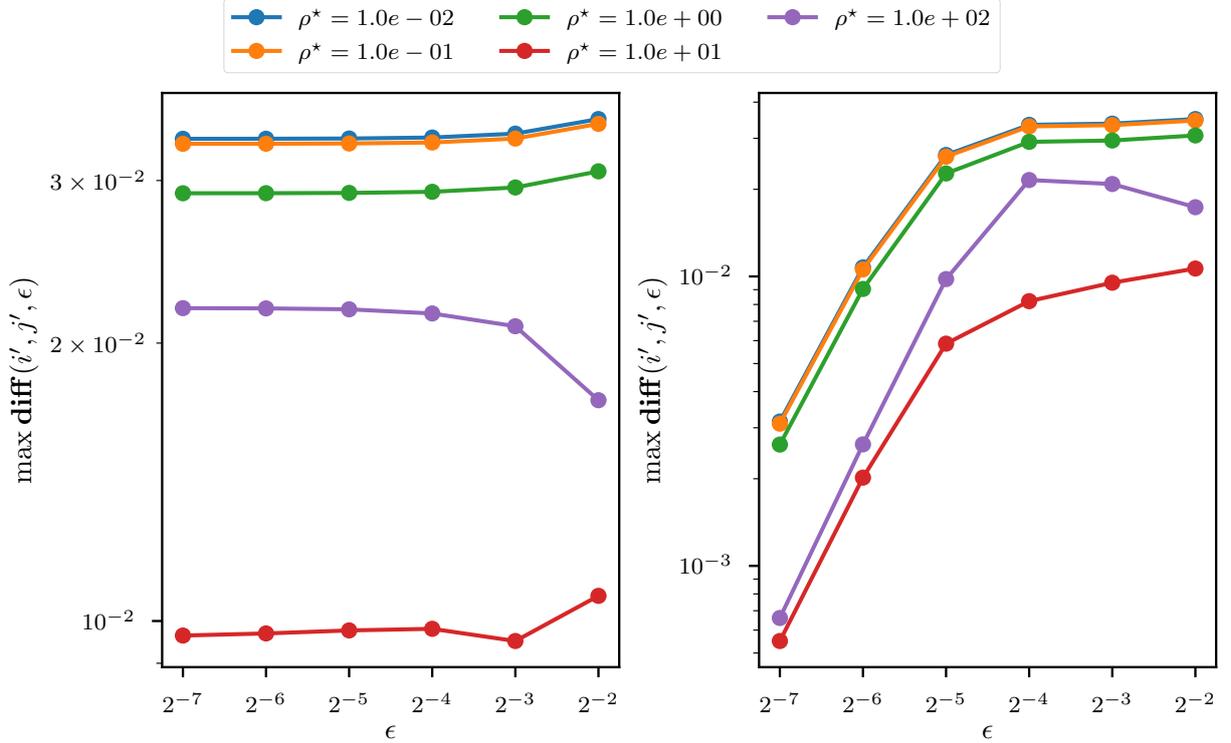}}
  \caption{Examine the convergence of the discrete model to the homogenized model with different $\rho^\star$.
    (Left) The y-axis represents the maximal $\mathbf{diff}(i',j',\epsilon)$ with $(i',j') \in \mathscr{F}_N^\circ$.
    (Right) The y-axis represents the maximal $\mathbf{diff}(i',j',\epsilon)$ with $(i',j') \in \mathscr{F}_N^\circ$ and $\abs{i'}+\abs{j'}\leq 10$.
  }
  \label{fig: errors}
\end{figure}

Kumar and McDowell suggested a special treatment for the $c_{\SSSText{D}}$ entry in $S_{\SSSText{D}}$ in obtaining the homogenized model \cite{Kumar2004}.
They proposed retaining the $\bigO(\epsilon^2)$ terms in $c_{\SSSText{D}}$, resulting in a different homogenization model that contains $\epsilon$ as an additional parameter.
For distinction, we denote this model as
\[
  S_{\SSSText{KM}}[i',j'] = \begin{bmatrix}
    A_{\SSSText{KM}}[i',j']                 & \bm{b}_{\SSSText{KM}}[i',j'] \\
    -\bm{b}_{\SSSText{KM}}^\intercal[i',j'] & c_{\SSSText{KM}}[i',j']
  \end{bmatrix}
\]
for the $(i',j')$ Fourier mode.
Except for the $c_{\SSSText{KM}}$ term, all other entries in $S_{\SSSText{KM}}$ are the same as those in $S_{\SSSText{C}}$ derived from the low-frequency principle.
Specifically, we have
\[
  c_{\SSSText{KM}} = 12 - 8\epsilon^2\pi^2 (i')^2 + 12 - 8\epsilon^2\pi^2 (i'+j')^2 + 12 - 8\epsilon^2\pi^2 (j')^2.
\]
We then numerically examine the convergence of the discrete model to the Kumar--McDowell homogenized model, presenting the results in \cref{fig: km errors}.
Similar to \cref{fig: errors}, we observe that for all Fourier modes, the convergence cannot hold as $\epsilon$ approaches zero.
The convergence pronounces for low-frequency modes, as shown in the right plot of \cref{fig: km errors}.

\begin{figure}[!ht]
  \centering
  \resizebox{\textwidth}{!}{\input{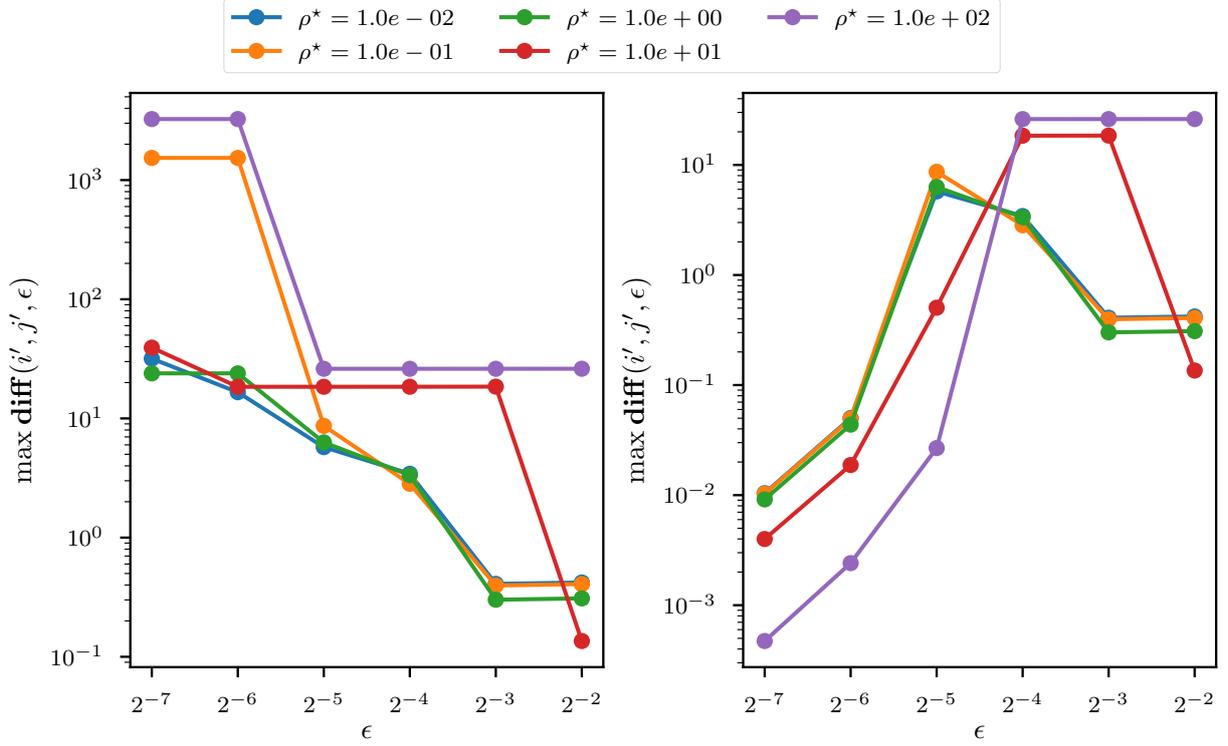}}
  \caption{Examine the convergence of the discrete model to the Kumar--McDowell homogenized model with different $\rho^\star$.
    (Left) The y-axis represents the maximal $\mathbf{diff}(i',j',\epsilon)$ with $(i',j') \in \mathscr{F}_N^\circ$.
    (Right) The y-axis represents the maximal $\mathbf{diff}(i',j',\epsilon)$ with $(i',j') \in \mathscr{F}_N^\circ$ and $\abs{i'}+\abs{j'}\leq 10$.
  }
  \label{fig: km errors}
\end{figure}

The observations above indicate an intrinsic difference between beam lattice homogenization and periodic homogenization.
When given general $L^2$-type right-hand terms $\bm{f}$ and $\tau$, the discrete model does not converge to the homogenized model as $\epsilon$ approaches zero.
Because in the discrete model, each node possesses two different physical variables: displacement and rotation.
The relationship between displacement and rotation is not akin to that of a spring.
Consequently, the homogenized model is not a Laplacian-type PDE, as seen in periodic homogenization.

\section{Quantitative homogenization}\label{sec: quantitative}
In this section, when comparing two non-negative quantities $a$ and $b$, we introduce a notation $a \lesssim_p b$ to represent the relation $a \leq C b$, where $C$ is a positive constant depending on the parameter $p$.
The notation $a \gtrsim_p b$ is defined similarly, and $a \approx_p b$ means that $a \lesssim_p b$ and $a \gtrsim_p b$.

Considering that $\hat{\bm{u}}_{\SSSText{D}}$ and $\hat{\theta}_{\SSSText{D}}$ are two different physical variables, we consider a field-splitting formulation for \cref{eq: D balance}.
Precisely, we obtain that
\begin{align}
  \underbrace{\RoundBrackets*{A_{\SSSText{D}}+\bm{b}_{\SSSText{D}}\otimes \bm{b}_{\SSSText{D}}/c_{\SSSText{D}}}}_{\coloneqq B_{\SSSText{D}}} \hat{\bm{u}}_{\SSSText{D}} & = \hat{\bm{f}}-\bm{b}_{\SSSText{D}}\hat{\tau}/c_{\SSSText{D}},                                     \label{eq: D u}     \\
  \hat{\theta}_{\SSSText{D}}                                                                                                                                            & = \hat{\tau}/c_{\SSSText{D}}+\bm{b}_{\SSSText{D}}\cdot \hat{\bm{u}}_{\SSSText{D}}/c_{\SSSText{D}}. \label{eq: D theta}
\end{align}
Similarly, we can derive an analogous formulation from \cref{eq: C balance}:
\begin{align}
  \underbrace{\RoundBrackets*{A_{\SSSText{C}}+\bm{b}_{\SSSText{C}}\otimes \bm{b}_{\SSSText{C}}/c_{\SSSText{C}}}}_{\coloneqq B_{\SSSText{C}}} \hat{\bm{u}}_{\SSSText{C}} & = \hat{\bm{f}}-\bm{b}_{\SSSText{C}}\hat{\tau}/c_{\SSSText{C}},                                     \label{eq: C u}     \\
  \hat{\theta}_{\SSSText{C}}                                                                                                                                            & = \hat{\tau}/c_{\SSSText{C}}+\bm{b}_{\SSSText{C}}\cdot \hat{\bm{u}}_{\SSSText{C}}/c_{\SSSText{C}}. \label{eq: C theta}
\end{align}
Here $B_{\SSSText{D}}$ and $B_{\SSSText{C}}$ are also the Schur complements (Ref.\ \cite{Zhang2005}) of $S_{\SSSText{D}}$ and $S_{\SSSText{C}}$, respectively.
Our strategy starts to establish a quantitative estimate on $B_{\SSSText{D}}^{-1} -B_{\SSSText{C}}^{-1}$.

We can first observe that $B_{\SSSText{D}}$ and $B_{\SSSText{C}}$ are both real symmetric matrices.
It is natural to consider the quadratic forms associated with $B_{\SSSText{D}}$ and $B_{\SSSText{C}}$.
For any $\bm{v} \in \Real^2$, we have
\begin{align*}
  \bm{v} \cdot B_{\SSSText{D}} \bm{v} & = 4\rho^\star\RoundBrackets*{\abs{\frac{\sin(\pi i' \epsilon)}{\epsilon}\bm{l}_1\cdot\bm{v}}^2  +  \abs{\frac{\sin(\pi(i'+j')\epsilon)}{\epsilon}\bm{l}_2\cdot\bm{v}}^2 + \abs{\frac{\sin(\pi j' \epsilon)}{\epsilon}\bm{l}_3\cdot \bm{v}}^2}             \\
                                      & \quad +48\RoundBrackets*{\abs{\frac{\sin(\pi i' \epsilon)}{\epsilon}\bm{l}_1^\perp\cdot\bm{v}}^2 + \abs{\frac{\sin(\pi(i'+j')\epsilon)}{\epsilon}\bm{l}_2^\perp\cdot\bm{v}}^2 + \abs{\frac{\sin(\pi j' \epsilon)}{\epsilon}\bm{l}_3^\perp\cdot \bm{v}}^2} \\
                                      & \quad - \frac{24^2}{c_{\SSSText{D}}}\bigg\lvert \cos(\pi i' \epsilon) \frac{\sin(\pi i' \epsilon)}{\epsilon}\bm{l}_1^\perp\cdot \bm{v}+ \cos(\pi(i'+j')\epsilon) \frac{\sin(\pi (i'+j') \epsilon)}{\epsilon}\bm{l}_2^\perp\cdot \bm{v}                    \\
                                      & \quad \quad \quad \quad \quad \quad + \cos(\pi j' \epsilon) \frac{\sin(\pi j' \epsilon)}{\epsilon}\bm{l}_3^\perp\cdot \bm{v}\bigg\rvert^2
\end{align*}
and
\begin{align*}
  \bm{v} \cdot B_{\SSSText{C}} \bm{v} & = 4\rho^\star\RoundBrackets*{\abs{\pi i' \bm{l}_1\cdot\bm{v}}^2  + \abs{\pi(i'+j')\bm{l}_2\cdot\bm{v}}^2 + \abs{\pi j'\bm{l}_3\cdot \bm{v}}^2}             \\
                                      & \quad +48\RoundBrackets*{\abs{\pi i'\bm{l}_1^\perp\cdot\bm{v}}^2 + \abs{\pi(i'+j')\bm{l}_2^\perp\cdot\bm{v}}^2 + \abs{\pi j'\bm{l}_3^\perp\cdot \bm{v}}^2} \\
                                      & \quad - \frac{24^2}{c_{\SSSText{C}}}\abs{\pi i'\bm{l}_1^\perp\cdot \bm{v}+\pi(i'+j')\bm{l}_2^\perp\cdot \bm{v} + \pi j'\bm{l}_3^\perp\cdot \bm{v}}^2.
\end{align*}
It may not be immediately clear that $B_{\SSSText{D}}$ and $B_{\SSSText{C}}$ are positive definite due to the presence of negative terms.
However, we can establish an identity in the following lemma which provides a positive answer to this question.
\begin{lemma}\label{lem: strange id}
  For  vectors $\bm{w}_1,\bm{w}_2,\dots,\bm{w}_n$ and angles $\theta_1,\theta_2,\dots,\theta_n$, it holds that for any $\bm{v}$,
  \begin{align}
           & \quad \sum_{t=1}^n \abs{\bm{w}_t \cdot \bm{v}}^2 - \frac{12}{c} \abs{\sum_{t=1}^n \cos \theta_t \bm{w}_t \cdot \bm{v}}^2                                                                                                                                 \\ \nonumber
    \qquad & = \frac{4}{c}\RoundBrackets*{\sum_{t'=1}^n \sin^2 \theta_{t'}} \sum_{t=1}^n \abs{\bm{w}_{t}\cdot \bm{v}}^2 + \frac{6}{c} \sum_{t=1}^n \sum_{t'}^{t'\neq t} \abs{(\cos\theta_{t'}\bm{w}_t-\cos \theta_{t}\bm{w}_{t'})\cdot \bm{v}}^2, \label{eq: good id}
  \end{align}
  where $c = \sum_{t=1}^n \RoundBrackets*{4 + 8 \cos^2 \theta_t}=\sum_{t=1}^n \RoundBrackets*{12 - 8 \sin^2 \theta_t}$.
\end{lemma}
\begin{proof}
  We first expand the term $\abs{\sum_{t=1}^n \cos \theta_t \bm{w}_t \cdot \bm{v}}^2$, which gives
  \[
    \sum_{t=1}^n \cos^2 \theta_t \abs{\bm{w}_t \cdot \bm{v}}^2 + \sum_{t=1}^n \sum_{t'}^{t'\neq t} \cos \theta_t \cos \theta_{t'} \bm{v}\cdot \bm{w}_t \otimes \bm{w}_{t'}  \bm{v}.
  \]
  The coefficient before $\abs{\bm{w}_t\cdot \bm{v}}^2$ can be simplified as
  \begin{align*}
    1-\frac{12\cos^2\theta_t}{c} & = \frac{4-4\cos^2\theta_t}{c}+\sum_{t'}^{t'\neq t} \frac{4+8\cos^2\theta_{t'}}{c}                                     \\
                                 & = \frac{4}{c}\sin^2\theta_t + \frac{1}{c}\sum_{t'}^{t'\neq t} \RoundBrackets*{12\cos^2\theta_{t'}+4\sin^2\theta_{t'}} \\
                                 & = \frac{4}{c}\sum_{t'=1}^n \sin^2 \theta_{t'} + \frac{12}{c} \sum_{t'}^{t'\neq t} \cos^2 \theta_{t'}.
  \end{align*}
  Next, We check the remaining term, which leads to
  \begin{align*}
     & \quad \sum_{t=1}^n \sum_{t'}^{t'\neq t} \cos^2 \theta_{t'}\abs{\bm{w}_t \cdot \bm{v}}^2-\cos \theta_t \cos \theta_{t'} \bm{v}\cdot \bm{w}_t \otimes \bm{w}_{t'}  \bm{v} \\
     & =\frac{1}{2} \sum_{t=1}^n \sum_{t'}^{t'\neq t} \abs{(\cos\theta_{t'}\bm{w}_t-\cos \theta_{t}\bm{w}_{t'})\cdot \bm{v}}^2.
  \end{align*}
  We hence establish the identity.
\end{proof}
By setting $n=3$ and $[\bm{w}_1,\ \bm{w}_2,\ \bm{w}_3]=[\sin(\pi i' \epsilon)/\epsilon \bm{l}_1^\perp,\sin(\pi (i'+j') \epsilon)/\epsilon \bm{l}_2^\perp,\sin(\pi j' \epsilon)/\epsilon \bm{l}_3^\perp]$ or $[\pi i'\bm{l}_1^\perp,\pi (i'+j')\bm{l}_2^\perp,\pi j'\bm{l}_3^\perp]$, we can derive that $B_{\SSSText{D}}$ and $B_{\SSSText{C}}$ are all positive semi-definite.
However, it is not enough for quantitative analysis.
In \cref{prop: coercive D}, we refine this result by providing an estimate on the minimal eigenvalue of $B_{\SSSText{D}}$.
Before reaching \cref{prop: coercive D}, we establish \cref{lem: coercive 1,lem: quadratic form estimate} to facilitate the proof.
\begin{lemma}\label{lem: coercive 1}
  For arbitrary 2D vectors $\bm{m}_1$ and $\bm{m}_2$ satisfying $\abs{\bm{m}_1}=\abs{\bm{m}_2}=1$ and $\bm{m}_1$ not parallel to $\bm{m}_2$, it holds that for all $\bm{v} \in \Real^2$,
  \begin{align*}
     & \frac{\sin^2(\pi i' \epsilon)}{\epsilon^2} \abs{\bm{m}_1 \cdot \bm{v}}^2 + \frac{\sin^2(\pi j' \epsilon)}{\epsilon^2} \abs{\bm{m}_2 \cdot \bm{v}}^2 + \abs{\RoundBrackets*{\frac{\sin(\pi i' \epsilon)}{\epsilon}\bm{m}_1^\perp- \frac{\sin(\pi j' \epsilon)}{\epsilon}\bm{m}_2^\perp}\cdot \bm{v}}^2 \\
     & \gtrsim_{\bm{m}_1, \bm{m}_2} \RoundBrackets*{(i')^2+(j')^2}\abs{\bm{v}}^2
  \end{align*}
  for all $(i',j') \in \mathscr{F}_N$.
\end{lemma}
\begin{proof}
  Writing $\bm{v}=z_1\bm{m}_2^\perp+z_2\bm{m}_1^\perp$, we can see that $\abs{\bm{v}}^2\approx_{\bm{m}_1,\bm{m}_2} z_1^2+z_2^2$.
  Therefore, the quadratic form can be rewritten in terms of $z_1$ and $z_2$, and we can derive the matrix representation as
  \[
    G=\begin{bmatrix}
      I^2+K^2 & KL        \\
      KL      & J^2 + L^2
    \end{bmatrix}
  \]
  where
  \begin{gather*}
    I^2 = \frac{\sin^2(\pi i' \epsilon)}{\epsilon^2}\abs{\bm{m}_1 \cdot \bm{m}_2^\perp}^2 \approx_{\bm{m}_1, \bm{m}_2} (i')^2, \\
    J^2 = \frac{\sin^2(\pi j' \epsilon)}{\epsilon^2}\abs{\bm{m}_2 \cdot \bm{m}_1^\perp}^2 \approx_{\bm{m}_1, \bm{m}_2} (j')^2, \\
    K = \frac{\sin(\pi i' \epsilon)}{\epsilon} \bm{m}_1^\perp \cdot \bm{m}_2^\perp-\frac{\sin(\pi j'\epsilon)}{\epsilon}, \quad L = \frac{\sin(\pi i' \epsilon)}{\epsilon}-\frac{\sin(\pi j' \epsilon)}{\epsilon}\bm{m}_2^\perp \cdot \bm{m}_1^\perp.
  \end{gather*}
  It is straightforward to verify that
  \[
    \lambda_{\mathup{min}}(G) \approx \frac{\det G}{\tr G} = \frac{I^2J^2+K^2J^2+L^2I^2}{I^2+J^2+K^2+L^2}.
  \]
  We can observe that $K^2 \lesssim_{\bm{m}_1,\bm{m}_2} (i')^2+(j')^2$ and $L^2 \lesssim_{\bm{m}_1,\bm{m}_2} (i')^2+(j')^2$.
  Therefore, we can derive an estimate on the denominator $I^2+J^2+K^2+L^2\approx_{\bm{m}_1,\bm{m}_2} (i')^2+(j')^2$.
  The numerator part requires a case-by-case analysis.
  Due to the symmetry, we only need to consider the case where $j'\geq 0$.
  \begin{itemize}
    \item ($0\leq j' \leq t_1 i'$) In this case, we can find a small enough positive $t_1$ such that there exists a positive $t_1'$ with
          \[
            L^2 \geq t_1' \frac{\sin^2(\pi i' \epsilon)}{\epsilon^2} \approx_{\bm{m}_1,\bm{m}_2} (i')^2.
          \]
          Therefore, we can derive that $L^2I^2\gtrsim_{\bm{m}_1,\bm{m}_2} (i')^4$, which leads to
          \[
            \frac{\lambda_{\mathup{min}}(G)}{(i')^2+(j')^2} \gtrsim_{\bm{m}_1,\bm{m}_2} 1.
          \]
    \item ($0 \leq i' \leq t_2 j'$) This case is similar to the previous one. We can find $t_2$ and $t_2'$ such that
          \[
            K^2 \geq t_2' \frac{\sin^2(\pi j' \epsilon)}{\epsilon^2} \approx_{\bm{m}_1,\bm{m}_2} (j')^2.
          \]
          Again, it holds that
          \[
            \frac{\lambda_{\mathup{min}}(G)}{(i')^2+(j')^2} \gtrsim_{\bm{m}_1,\bm{m}_2} 1.
          \]
    \item ($1/t_1 \leq i' \leq t_2 j'$) We can now directly control $I^2J^2$ as $(i')^4$ without touching $K^2J^2$ and $L^2I^2$, which leads to the desired estimate.
    \item ($i' \leq 0$) Techniques for this case are similar to the previous ones, and we hence omit the details.
  \end{itemize}

  Combining all results, we complete the proof for the estimate.
\end{proof}
\begin{lemma}\label{lem: quadratic form estimate}
  For arbitrary 2D vectors $\bm{m}_1$ and $\bm{m}_2$ satisfying $\abs{\bm{m}_1}=\abs{\bm{m}_2}=1$ and $\bm{m}_1$ not parallel to $\bm{m}_2$, it holds that for all $\bm{v} \in \Real^2$
  \begin{align*}
     & \quad \frac{\sin^2(\pi i' \epsilon)}{\epsilon^2} \abs{\bm{m}_1 \cdot \bm{v}}^2 + \frac{\sin^2(\pi j' \epsilon)}{\epsilon^2} \abs{\bm{m}_2 \cdot \bm{v}}^2                                                                                              \\
     & \qquad + \RoundBrackets*{\sin^2(\pi i' \epsilon)+\sin^2(\pi j' \epsilon)}\RoundBrackets*{\frac{\sin^2(\pi i' \epsilon)}{\epsilon^2} \abs{\bm{m}_1^\perp \cdot \bm{v}}^2+ \frac{\sin^2(\pi j' \epsilon)}{\epsilon^2}\abs{\bm{m}_2^\perp\cdot \bm{v}}^2} \\
     & \qquad + \abs{\RoundBrackets*{\cos(\pi j' \epsilon)\frac{\sin(\pi i' \epsilon)}{\epsilon}\bm{m}_1^\perp- \cos(\pi i' \epsilon)\frac{\sin(\pi j' \epsilon)}{\epsilon}\bm{m}_2^\perp}\cdot \bm{v}}^2                                                     \\
     & \gtrsim_{\bm{m}_1, \bm{m}_2} \RoundBrackets*{(i')^2+(j')^2}\abs{\bm{v}}^2
  \end{align*}
  for all $(i',j')\in \mathscr{F}_N$.
\end{lemma}
\begin{proof}
  We first set a positive $\delta$ which will be determined later.

  Suppose that $\sin^2(\pi i'\epsilon)+\sin^2(\pi j' \epsilon) \geq \delta$, we can obtain that
  \begin{align*}
     & \quad \frac{\sin^2(\pi i' \epsilon)}{\epsilon^2} \abs{\bm{m}_1 \cdot \bm{v}}^2 + \frac{\sin^2(\pi j' \epsilon)}{\epsilon^2} \abs{\bm{m}_2 \cdot \bm{v}}^2                                     \\
     & \qquad + \delta \RoundBrackets*{\frac{\sin^2(\pi i' \epsilon)}{\epsilon^2} \abs{\bm{m}_1^\perp \cdot \bm{v}}^2+ \frac{\sin^2(\pi j' \epsilon)}{\epsilon^2}\abs{\bm{m}_2^\perp\cdot \bm{v}}^2} \\
     & \gtrsim_{\bm{m}_1, \bm{m}_2} \min\CurlyBrackets{1,\delta} \RoundBrackets*{(i')^2+(j')^2}\abs{\bm{v}}^2,
  \end{align*}
  where we implicitly utilize the fact that $\abs{\bm{m}_1 \cdot \bm{v}}^2+\abs{\bm{m}_1^\perp \cdot \bm{v}}^2 \approx_{\bm{m}_1} \abs{\bm{v}}^2$ and $\abs{\bm{m}_2 \cdot \bm{v}}^2+\abs{\bm{m}_2^\perp \cdot \bm{v}}^2 \approx_{\bm{m}_2} \abs{\bm{v}}^2$.
  The target estimate holds in this case.

  If $\sin^2(\pi i'\epsilon)+\sin^2(\pi j' \epsilon) \leq \delta$, we have $\abs{1-\cos(\pi i' \epsilon)}\lesssim \delta$ and $\abs{1-\cos(\pi j' \epsilon)}\lesssim \delta$.
  We can then show that
  \begin{align*}
     & \quad \abs{\RoundBrackets*{\cos(\pi j' \epsilon)\frac{\sin(\pi i' \epsilon)}{\epsilon}\bm{m}_1^\perp- \cos(\pi i' \epsilon)\frac{\sin(\pi j' \epsilon)}{\epsilon}\bm{m}_2^\perp}\cdot \bm{v}}^2                                       \\
     & \gtrsim \abs{\RoundBrackets*{\frac{\sin(\pi i' \epsilon)}{\epsilon}\bm{m}_1^\perp- \frac{\sin(\pi j' \epsilon)}{\epsilon}\bm{m}_2^\perp}\cdot \bm{v}}^2                                                                               \\
     & \quad - \abs{\cos(\pi j' \epsilon)-1}^2\frac{\sin^2(\pi i' \epsilon)}{\epsilon^2} \abs{\bm{m}_1^\perp \cdot \bm{v}}^2 - \abs{\cos(\pi i' \epsilon)-1}^2\frac{\sin^2(\pi j' \epsilon)}{\epsilon^2} \abs{\bm{m}_2^\perp \cdot \bm{v}}^2 \\
     & \gtrsim_{\bm{m}_1,\bm{m}_2} \abs{\RoundBrackets*{\frac{\sin(\pi i' \epsilon)}{\epsilon}\bm{m}_1^\perp- \frac{\sin(\pi j' \epsilon)}{\epsilon}\bm{m}_2^\perp}\cdot \bm{v}}^2 - \delta^2 \RoundBrackets*{(i')^2+(j')^2}\abs{\bm{v}}^2.
  \end{align*}
  Combining \cref{lem: coercive 1}, we can choose $\delta$ small enough such that the desired estimate holds.
\end{proof}
\begin{proposition}\label{prop: coercive D}
  It holds that
  \[
    \lambda_{\mathup{min}}(B_{\SSSText{D}}[i',j']) \gtrsim_{\rho^\star} (i')^2+(j')^2
  \]
  for all $(i', j')\in \mathscr{F}_N$
\end{proposition}
The proof of \cref{prop: coercive D} is straightforward by following \cref{lem: coercive 1,lem: quadratic form estimate}, and we hence omit the details here.
\Cref{prop: coercive C} serves as a counterpart to \cref{prop: coercive D}, which shows that $B_{\SSSText{C}}$ exhibits the same form of coercivity as $B_{\SSSText{C}}$.
\begin{proposition}\label{prop: coercive C}
  It holds that
  \[
    \lambda_{\mathup{min}}(B_{\SSSText{C}}[i',j']) \gtrsim_{\rho^\star} (i')^2+(j')^2
  \]
  for all $(i',j')\in \mathscr{F}_N$.
\end{proposition}
\begin{proof}
  We can mimic the proof of \cref{lem: coercive 1} to prove the following statement:
  For arbitrary 2D vectors $\bm{m}_1$ and $\bm{m}_2$ satisfying $\abs{\bm{m}_1}=\abs{\bm{m}_2}=1$ and $\bm{m}_1$ not parallel to $\bm{m}_2$, it holds that for all $\bm{v} \in \Real^2$
  \[
    (i')^2 \abs{\bm{m}_1 \cdot \bm{v}}^2 + (j')^2 \abs{\bm{m}_2 \cdot \bm{v}}^2 + \abs{\RoundBrackets*{i'\bm{m}_1^\perp- j'\bm{m}_2^\perp}\cdot \bm{v}}^2  \gtrsim_{\bm{m}_1, \bm{m}_2} \RoundBrackets*{(i')^2+(j')^2}\abs{\bm{v}}^2
  \]
  for any $(i',j')$.
  We can then complete the proof by taking $\bm{m}_1=\bm{l}_1$ and $\bm{m}_2 = \bm{l}_3$.
\end{proof}
Our strategy to estimate $B_{\SSSText{D}}^{-1}-B_{\SSSText{C}}^{-1}$ now becomes clear.
By utilizing the basic inequality of matrix norms, we have
\begin{equation}\label{eq: basic inequality}
  \norm{B_{\SSSText{D}}^{-1}-B_{\SSSText{C}}^{-1}} \leq \norm{B_{\SSSText{D}}-B_{\SSSText{C}}} \norm{B_{\SSSText{D}}^{-1}} \norm{B_{\SSSText{C}}^{-1}} \lesssim_{\rho^\star} \frac{\norm{B_{\SSSText{D}}-B_{\SSSText{C}}}}{\RoundBrackets*{(i')^2+(j')^2}^2}.
\end{equation}
Then, our target reduces to estimating $\norm{B_{\SSSText{D}}-B_{\SSSText{C}}}$.
Recalling that all norms on $2$-by-$2$ matrices are equivalent, we hence only need to estimate the entry-wise difference of $B_{\SSSText{D}}$ and $B_{\SSSText{C}}$.
The following inequality offers a basic tool to achieve this goal.
\begin{lemma}\label{lem: sin cos}
  For any non-negative integers $m$, $n$, $m^\circ$ and $n^\circ$, it holds that
  \begin{align*}
     & \quad \abs{\frac{\sin^m(\pi i' \epsilon) \sin^n(\pi j' \epsilon)}{\epsilon^{m+n}}\cos^{m^\circ}(\pi i' \epsilon)\cos^{n^\circ}(\pi j' \epsilon)-(\pi i')^m(\pi j')^n} \\
     & \lesssim_{m, n, m^\circ, n^\circ}
    \begin{cases}
      \epsilon^2(j')^{n+2},                                              & \text{ if } m=m^\circ=0, \text{ and } n \geq 1, \\
      \epsilon^2(i')^{m+2},                                              & \text{ if } m \geq 1, \text{ and } n=n^\circ=0, \\
      \epsilon^2\SquareBrackets*{(i')^{m+2}(j')^{n}+(i')^{m}(j')^{n+2}}, & \text{ else},                                   \\
    \end{cases}
  \end{align*}
  for any $i'$ and $j'$ satisfying $i'\geq 0$, $j'\geq 0$, and $i'+j'\geq 1$.
\end{lemma}
\begin{proof}
  We here only consider the last case.
  Split the target term into four parts as
  \begin{align*}
     & \quad \abs{\frac{\sin^m(\pi i' \epsilon) \sin^n(\pi j' \epsilon)}{\epsilon^{m+n}}\cos^{m^\circ}(\pi i' \epsilon)\cos^{n^\circ}(\pi j' \epsilon)-(\pi i')^m(\pi j')^n}                       \\
     & \leq \abs{\RoundBrackets*{\frac{\sin^m(\pi i' \epsilon)}{\epsilon^m} - (\pi i')^m}\frac{\sin^n(\pi j' \epsilon)}{\epsilon^n}\cos^{m^\circ}(\pi i' \epsilon)\cos^{n^\circ}(\pi j' \epsilon)} \\
     & \quad + \abs{(\pi i')^m\RoundBrackets*{\frac{\sin^n(\pi j' \epsilon)}{\epsilon^n} - (\pi j')^n}\cos^{m^\circ}(\pi i' \epsilon)\cos^{n^\circ}(\pi j' \epsilon)}                              \\
     & \quad + \abs{(\pi i')^m (\pi j')^n \RoundBrackets*{\cos^{m^\circ}(\pi i' \epsilon)-1}\cos^{n^\circ}(\pi j' \epsilon)}                                                                       \\
     & \quad + \abs{(\pi i')^m (\pi j')^n \RoundBrackets*{\cos^{n^\circ}(\pi j' \epsilon)-1}}                                                                                                      \\
     & \coloneqq K_1 + K_2 + K_3 + K_4.
  \end{align*}
  It is easy to verify that $\abs{\sin^m(\pi i' \epsilon)/\epsilon^m - (\pi i')^m}\lesssim_m \epsilon^2 (i')^{m+2}$ and $\abs{\sin^n(\pi j' \epsilon)/\epsilon^n} \lesssim (j')^n$.
  Therefore, we can derive that $K_1 \lesssim_m \epsilon^2 (i')^{m+2}(j')^n$.
  Similarly, we can obtain that $K_2 \lesssim_n \epsilon^2(i')^m(j')^{n+2}$.
  For $K_3$ and $K_4$, we can establish the estimates that $\abs{\cos^{m^\circ}(\pi i' \epsilon)-1}\lesssim_{m^\circ} \epsilon^2 (i')^2$ and $\abs{\cos^{n^\circ}(\pi j' \epsilon)-1}\lesssim_{n^\circ} \epsilon^2 (j')^2$.
  We hence complete the proof by summing up all estimates.
\end{proof}
We here present the estimate of $\norm{B_{\SSSText{D}}^{-1}[i',j']-B_{\SSSText{C}}^{-1}[i',j']}$.
\begin{theorem} \label{thm: inverse difference}
  For any $(i',j')\in \mathscr{F}_N^\circ$, it holds that
  \[
    \norm{B_{\SSSText{D}}^{-1}[i',j']-B_{\SSSText{C}}^{-1}[i',j']}\lesssim_{\rho^\star} \epsilon^2.
  \]
\end{theorem}
\begin{proof}
  As discussed, we now focus on the entry-wise difference of $B_{\SSSText{D}}$ and $B_{\SSSText{C}}$.
  Recalling the construction of $B_{\SSSText{D}}$ and $B_{\SSSText{C}}$, we can rewrite $B_{\SSSText{D}}-B_{\SSSText{C}}$ as the following expression:
  \begin{align*}
    B_{\SSSText{D}}-B_{\SSSText{C}} & =\RoundBrackets*{A_{\SSSText{1,D}}-A_{\SSSText{1,C}}}+\RoundBrackets*{A_{\SSSText{2,D}}-A_{\SSSText{2,C}}}+\RoundBrackets*{A_{\SSSText{3,D}}-A_{\SSSText{3,C}}}                                                                                                  \\
                                    & \quad + \RoundBrackets*{\bm{b}_{\SSSText{D}}\otimes \bm{b}_{\SSSText{D}}-\bm{b}_{\SSSText{C}}\otimes \bm{b}_{\SSSText{C}}}/c_{\SSSText{D}}+\bm{b}_{\SSSText{C}}\otimes \bm{b}_{\SSSText{C}}\RoundBrackets*{\frac{1}{c_{\SSSText{D}}}-\frac{1}{c_{\SSSText{C}}}}.
  \end{align*}
  Thanks to \cref{lem: sin cos}, for $x$-mode, it is easy to see that
  \[
    \norm{A_{\SSSText{1,D}}-A_{\SSSText{1,C}}} \lesssim_{\rho^\star} \epsilon^2(i')^4\text{ and } \norm{A_{\SSSText{3,D}}-A_{\SSSText{3,C}}} \lesssim_{\rho^\star} \epsilon^2(j')^4.
  \]
  For $y$-mode, before applying \cref{lem: sin cos}, we can first take an expansion as
  \begin{multline*}
    \sin^2(\pi (i'+j')\epsilon)=\sin^2(\pi i' \epsilon)\cos^2(\pi j' \epsilon)+\sin^2(\pi j' \epsilon)\cos^2(\pi i' \epsilon)\\
    +2 \sin(\pi i' \epsilon)\sin(\pi j' \epsilon)\cos(\pi i' \epsilon)\cos(\pi j' \epsilon).
  \end{multline*}
  Then, we can derive that
  \[
    \norm{A_{\SSSText{2,D}}-A_{\SSSText{2,C}}} \lesssim_{\rho^\star} \epsilon^2\RoundBrackets*{(i')^4+(j')^4+\abs{i'}^3\abs{j'}+\abs{i'}\abs{j'}^3} \lesssim_{\rho^\star} \RoundBrackets*{(i')^2+(j')^2}^2.
  \]
  The procedure for handling $\bm{b}_{\SSSText{D}}\otimes \bm{b}_{\SSSText{D}}-\bm{b}_{\SSSText{C}}\otimes \bm{b}_{\SSSText{C}}$ is similar, and we first expand the term as
  \[
    \bm{b}_{\SSSText{D}}\otimes \bm{b}_{\SSSText{D}}-\bm{b}_{\SSSText{C}}\otimes \bm{b}_{\SSSText{C}} = \sum_{\imath \in \CurlyBrackets{1,2,3}} \sum_{\jmath \in \CurlyBrackets{1,2,3}}\bm{b}_{\SSSText{\imath,D}} \otimes \bm{b}_{\SSSText{\jmath, D}}-\bm{b}_{\SSSText{\imath, C}}\otimes \bm{b}_{\SSSText{\jmath, C}}.
  \]
  For the case $\imath=1$ and $\jmath=2$, we can derive that
  \begin{align*}
     & \quad \frac{\sin(2\pi i' \epsilon)\sin(2\pi (i'+j') \epsilon)}{\epsilon^2} =\frac{4}{\epsilon^2}\sin(\pi i' \epsilon)\cos(\pi i' \epsilon)                                                                         \\
     & \qquad \qquad \times \big(\sin(\pi i' \epsilon)\cos(\pi j' \epsilon)+\cos(\pi i' \epsilon)\sin(\pi j' \epsilon)\big)                                                                                               \\
     & \qquad \qquad \times \big(\cos(\pi i' \epsilon)\cos(\pi j' \epsilon)-\sin(\pi i' \epsilon)\sin(\pi j' \epsilon)\big)                                                                                               \\
     & = 4 \frac{\sin^2(\pi i' \epsilon)}{\epsilon^2}\cos^2(\pi i' \epsilon)\cos^2(\pi j' \epsilon)+4\frac{\sin(\pi i' \epsilon)\sin(\pi j' \epsilon)}{\epsilon^2}\cos^3(\pi i' \epsilon)\cos(\pi j' \epsilon)            \\
     & \quad - 4 \frac{\sin^3(\pi i' \epsilon)\sin(\pi j' \epsilon)}{\epsilon^2} \cos(\pi i' \epsilon)\cos(\pi j' \epsilon)- 4 \frac{\sin^2(\pi i' \epsilon)\sin^2(\pi j' \epsilon)}{\epsilon^2} \cos^2(\pi i' \epsilon).
  \end{align*}
  Therefore, we can show that
  \begin{align*}
     & \quad \abs{\frac{\sin(2\pi i' \epsilon)\sin(2\pi (i'+j') \epsilon)}{\epsilon^2}-4\pi^2 i'(i'+j')}                                                       \\
     & \lesssim \epsilon^2 \RoundBrackets*{(i')^4+(i')^2(j')^2+\abs{i'}^3\abs{j'}+\abs{i'}\abs{j'}^3}                                                          \\
     & \quad + \abs{\frac{\sin^3(\pi i' \epsilon)\sin(\pi j' \epsilon)}{\epsilon^2}} + \abs{\frac{\sin^2(\pi i' \epsilon)\sin^2(\pi j' \epsilon)}{\epsilon^2}} \\
     & \lesssim \epsilon^2 \RoundBrackets*{(i')^4+(i')^2(j')^2+\abs{i'}^3\abs{j'}+\abs{i'}\abs{j'}^3} \lesssim \epsilon^2 \RoundBrackets*{(i')^2+(j')^2}^2.
  \end{align*}
  Following the same procedure for other $(\imath, \jmath)$ and recalling that $c_{\SSSText{D}}\approx 1$, we can obtain that
  \[
    \norm{\RoundBrackets*{\bm{b}_{\SSSText{D}}\otimes \bm{b}_{\SSSText{D}}-\bm{b}_{\SSSText{C}}\otimes \bm{b}_{\SSSText{C}}}/c_{\SSSText{D}}} \lesssim \epsilon^2 \RoundBrackets*{(i')^2+(j')^2}^2.
  \]
  For the last term, we can directly calculate that $\norm{\bm{b}_{\SSSText{C}}\otimes \bm{b}_{\SSSText{C}}} \lesssim (i')^2+(j')^2$ and
  \begin{multline*}
    \abs{\frac{1}{c_{\SSSText{D}}}-\frac{1}{c_{\SSSText{C}}} } \lesssim \abs{c_{\SSSText{D}}-c_{\SSSText{C}}} \lesssim \abs{\sin^2(\pi i' \epsilon)}+ \abs{\sin^2(\pi (i'+j') \epsilon)}+\abs{\sin^2(\pi j' \epsilon)}\\
    \lesssim \epsilon^2\RoundBrackets*{(i')^2+(j')^2}.
  \end{multline*}
  Now we can arrive at
  \[
    \norm{B_{\SSSText{D}}[i',j']-B_{\SSSText{C}}[i',j']} \lesssim_{\rho^\star} \epsilon^2\RoundBrackets*{(i')^2+(j')^2}^2.
  \]
  Consequently, the proof is completed by combining the above estimate with \cref{eq: basic inequality}.
\end{proof}

We can simply obtain a corollary from \cref{thm: inverse difference}.
Suppose that $\hat{\tau}[i',j']=0$ for all $(i',j')\in \mathscr{F}_N^\circ$.
Then, we can show that
\begin{equation} \label{eq: mode difference}
  \abs{\hat{\bm{u}}_{\SSSText{D}}[i',j']- \hat{\bm{u}}_{\SSSText{C}}[i',j']}= \abs{\RoundBrackets*{B_{\SSSText{D}}^{-1}-B_{\SSSText{C}}^{-1}}\hat{\bm{f}}[i',j']} \lesssim_{\rho^\star} \epsilon^2 \abs{\hat{\bm{f}}[i',j']}.
\end{equation}
Therefore, we can rewrite \cref{eq: mode difference} as follows: if $\tau=0$, then
\[
  \norm{\bm{u}_{\SSSText{D}}-\bm{u}_{\SSSText{C}}}_{0} \lesssim_{\rho^\star} \epsilon^2 \norm{\bm{f}}_{0}.
\]
We then investigate the case that $\tau \neq 0$.

As we have explained, we cannot expect an estimate in \cref{eq: l2 convergence wrong}.
Hence, our strategy is to provide a higher regularity for $\tau$.
The following lemma guides us to establish a proper assumption on $\tau$.
\begin{lemma}\label{lem: diff B}
  For any $(i',j')\in \mathscr{F}_N$, it holds that
  \[
    \abs{\bm{b}_{\SSSText{D}}/c_{\SSSText{D}}-\bm{b}_{\SSSText{C}}/c_{\SSSText{C}}} \lesssim \epsilon^2 \RoundBrackets*{\abs{i'}^3+\abs{j'}^3}.
  \]
\end{lemma}
\begin{proof}
  The proof here follows the same approach as in \cref{thm: inverse difference} by using a decomposition:
  \[
    \abs{\bm{b}_{\SSSText{D}}/c_{\SSSText{D}}-\bm{b}_{\SSSText{C}}/c_{\SSSText{C}}} \leq \sum_{\imath \in \CurlyBrackets{1,2,3}} \abs{\bm{b}_{\SSSText{\imath, D}}-\bm{b}_{\SSSText{\imath, C}}}/c_{\SSSText{D}}+\abs{\bm{b}_{\SSSText{C}}} \abs{\frac{1}{c_{\SSSText{D}}}-\frac{1}{c_{\SSSText{C}}}}.
  \]
  As previously shown, we have $\sum_{\imath \in \CurlyBrackets{1,2,3}} \abs{\bm{b}_{\SSSText{\imath, D}}-\bm{b}_{\SSSText{\imath, C}}} \lesssim \epsilon^2 \RoundBrackets{\abs{i'}^3+\abs{j'}^3}$.
  Additionally, it has already been proved that $\abs{1/c_{\SSSText{D}}-1/c_{\SSSText{C}}} \lesssim \epsilon^2 \RoundBrackets{\abs{i'}^2+\abs{j'}^2}$.
  Therefore, the proof is completed.
\end{proof}
Now the materials for proving the main statement are prepared, and we observe that
\[
  \hat{\bm{u}}_{\SSSText{D}}-\hat{\bm{u}}_{\SSSText{C}} = \underbrace{\RoundBrackets*{B_{\SSSText{D}}^{-1}-B_{\SSSText{C}}^{-1}}\hat{\bm{f}}}_{\coloneqq J_1} + \underbrace{\RoundBrackets*{B_{\SSSText{C}}^{-1}-B_{\SSSText{D}}^{-1}} \bm{b}_{\SSSText{D}}/c_{\SSSText{D}} \hat{\tau}}_{\coloneqq J_2} + \underbrace{B_{\SSSText{C}}^{-1}\RoundBrackets*{\bm{b}_{\SSSText{C}}/c_{\SSSText{C}}-\bm{b}_\SSSText{D}/c_{\SSSText{D}}}\hat{\tau}}_{\coloneqq J_3}.
\]
It has been shown in \cref{eq: mode difference} that $\abs{J_1} \lesssim_{\rho^\star} \epsilon^2 \norm{\hat{\bm{f}}}_{0}$.
For $J_2$, it is easy to check that $\abs{\bm{b}_{\SSSText{D}}/c_{\SSSText{D}}} \lesssim \RoundBrackets*{\abs{i'}+\abs{j'}}$, and we can hence derive that $\abs{J_2} \lesssim_{\rho^\star}\epsilon^2\RoundBrackets*{\abs{i'}+\abs{j'}}\abs{\hat{\tau}}$.
The last term, $J_3$, can be estimated by \cref{prop: coercive C,lem: diff B}, giving $\abs{J_3} \lesssim_{\rho^\star} \epsilon^2 \RoundBrackets*{\abs{i'}+\abs{j'}}\abs{\hat{\tau}}$.
Therefore, we can conclude that
\[
  \abs{\hat{\bm{u}}_{\SSSText{D}}[i',j']-\hat{\bm{u}}_{\SSSText{C}}[i',j']} \lesssim_{\rho^\star} \epsilon^2 \CurlyBrackets*{\abs{\hat{\bm{f}}[i',j']}+\RoundBrackets*{\abs{i'}+\abs{j'}}\abs{\hat{\tau}[i',j']}}.
\]
To keep the $\bigO(\epsilon^2)$ order, we can impose a higher regularity assumption on $\tau$ as $\abs{\tau}_1$ remaining bounded w.r.t.\ $\epsilon$.

The difference of $\theta_{\SSSText{D}}$ and $\theta_{\SSSText{C}}$ can be expressed as
\[
  \hat{\theta}_{\SSSText{D}} - \hat{\theta}_{\SSSText{C}} = \underbrace{\RoundBrackets*{\bm{b}_{\SSSText{D}}/c_{\SSSText{D}}-\bm{b}_{\SSSText{C}}/c_{\SSSText{C}}}\cdot \hat{\bm{u}}_{\SSSText{D}}}_{\coloneqq K_1} + \underbrace{\bm{b}_{\SSSText{C}}\cdot \RoundBrackets*{\hat{\bm{u}}_{\SSSText{D}}-\hat{\bm{u}}_{\SSSText{C}}}/c_{\SSSText{C}}}_{\coloneqq K_2}+\underbrace{\hat{\tau}\RoundBrackets*{1/c_{\SSSText{D}}-1/c_{\SSSText{C}}}}_{\coloneqq K_3}.
\]
By \cref{prop: coercive D}, we can derive that
\[
  \abs{\hat{\bm{u}}_{\SSSText{D}}} \leq \norm{B_{\SSSText{D}}^{-1}} \RoundBrackets*{\abs{\hat{\bm{f}}}+\abs{\bm{b}_{\SSSText{D}}/c_{\SSSText{D}}}\abs{\hat{\tau}}}\lesssim_{\rho^\star} \RoundBrackets*{\abs{i'}^2+\abs{j'}^2}^{-1} \CurlyBrackets*{\abs{\hat{\bm{f}}}+\RoundBrackets*{\abs{i'}+\abs{j'}}\abs{\hat{\tau}}}.
\]
Combining \cref{lem: diff B}, we immediately obtain that
\[
  \abs{K_1} \lesssim_{\rho^\star} \epsilon^2 \CurlyBrackets*{\RoundBrackets*{\abs{i'}+\abs{j'}}\abs{\hat{\bm{f}}}+\RoundBrackets*{\abs{i'}^2+\abs{j'}^2}\abs{\hat{\tau}}}.
\]
Recalling $\abs{\bm{b}_{\SSSText{C}}}\lesssim \abs{i'}+\abs{j'}$ and previous results on $\abs{\hat{\bm{u}}_{\SSSText{D}}-\hat{\bm{u}}_{\SSSText{C}}}$, we can similarly show that
\[
  \abs{K_2} \lesssim_{\rho^\star} \epsilon^2 \CurlyBrackets*{\RoundBrackets*{\abs{i'}+\abs{j'}}\abs{\hat{\bm{f}}}+\RoundBrackets*{\abs{i'}^2+\abs{j'}^2}\abs{\hat{\tau}}}.
\]
We have proved $\abs{1/c_{\SSSText{D}}-1/c_{\SSSText{C}}} \lesssim \epsilon^2 \RoundBrackets{\abs{i'}^2+\abs{j'}^2}$, which yields $\abs{K_3} \lesssim \epsilon^2 \RoundBrackets{\abs{i'}^2+\abs{j'}^2} \abs{\hat{\tau}}$.
Thus, we can establish that
\[
  \abs{\hat{\theta}_{\SSSText{D}}[i',j'] - \hat{\theta}_{\SSSText{C}}[i',j']} \lesssim_{\rho^\star} \epsilon^2 \CurlyBrackets*{\RoundBrackets*{\abs{i'}+\abs{j'}}\abs{\hat{\bm{f}}[i',j']}+\RoundBrackets*{\abs{i'}^2+\abs{j'}^2}\abs{\hat{\tau}[i',j']}}.
\]
Hence, to maintain the $\bigO(\epsilon^2)$ order on the rotation variable, $\abs{\bm{f}}_1$ and $\abs{\tau}_2$ need to be bounded w.r.t.\ $\epsilon$.

Recall that as discrete functions, $\bm{f}$ and $\tau$ always have finite $\norm{\cdot}_0$ and $\abs{\cdot}_k$ norms.
This contrasts with the continuous case, as we have mentioned, where the function with finite $\norm{\cdot}_0$ and $\abs{\cdot}_k$ norms belongs to $H^k(D)$, a Sobolev space that is a subspace of $L^2(D)$.
Therefore, our results are different from conventional numerical analysis of solving PDEs (e.g., Ref.\ \cite{Brenner2008}), where a regularity assumption is usually imposed.
In our case, the ground truth is the discrete model, and our goal is, in some sense, to approximate the discrete model with the continuous model.
We can now summarize the main results of this section as the following theorem.
\begin{theorem} \label{thm: main}
  Suppose that
  \[
    \begin{bmatrix}
      \hat{\bm{f}}[0,0] \\ \hat{\tau}[0,0]
    \end{bmatrix} \in \im \RoundBrackets*{S_{\SSSText{D}}[0,0]}.
  \]
  Then, there exist solutions $(\bm{u}_{\SSSText{D}}[i',j'],\theta_{\SSSText{D}}[i',j'])$ to \cref{eq: D balance} and $(\bm{u}_{\SSSText{C}}[i',j'],\theta_{\SSSText{C}}[i',j'])$ to \cref{eq: C balance} for all $(i',j')\in \mathscr{F}_N$ and uniquely for $(i',j')\in \mathscr{F}_N^\circ$.
  Additionally, assume $\hat{\bm{u}}_{\SSSText{D}}[0,0]=\hat{\bm{u}}_{\SSSText{C}}[0,0]$ and $\hat{\theta}_{\SSSText{D}}[0,0]=\hat{\theta}_{\SSSText{C}}[0,0]$.
  Then, the following estimates hold:
  \begin{align*}
    \norm{\bm{u}_{\SSSText{D}}-\bm{u}_{\SSSText{C}}}_{0} & \lesssim_{\rho^\star} \epsilon^2 \RoundBrackets*{\norm{\bm{f}}_{0}+\abs{\tau}_{1}}, \\
    \norm{\theta_{\SSSText{D}}-\theta_{\SSSText{C}}}_{0} & \lesssim_{\rho^\star} \epsilon^2 \RoundBrackets*{\abs{\bm{f}}_{1}+\abs{\tau}_{2}}.
  \end{align*}
\end{theorem}
Consider the following fact: for a discrete function $g$, it holds that
\[
  \abs{g}_1 \leq \max_{(i', j')\in \mathscr{F}_N} \sqrt{\abs{i'}^2+\abs{j'}^2} \norm{g}_0 \lesssim \epsilon^{-1} \norm{g}_0.
\]
We can immediately obtain a corollary from the above theorem.
\begin{corollary}
  Under the assumptions of \cref{thm: main}, it holds that
  \[
    \norm{\bm{u}_{\SSSText{D}}-\bm{u}_{\SSSText{C}}}_{0}+\abs{\bm{u}_{\SSSText{D}}-\bm{u}_{\SSSText{C}}}_{1} + \norm{\theta_{\SSSText{D}}-\theta_{\SSSText{C}}}_{0} \lesssim_{\rho^\star} \epsilon \RoundBrackets*{\norm{\bm{f}}_{0}+\abs{\tau}_{1}}.
  \]
\end{corollary}
Meanwhile, as mode differences, $\RoundBrackets{\hat{\bm{u}}_{\SSSText{D}}[i',j']-\hat{\bm{u}}_{\SSSText{C}}[i',j']}$ and $\RoundBrackets{\hat{\theta}_{\SSSText{D}}[i',j']-\hat{\theta}_{\SSSText{C}}[i',j']}$, have been explicitly estimated, we can also derive a corollary regarding the $H^k$ norm.
\begin{corollary}
  Under the assumptions of \cref{thm: main}, it holds that
  \begin{align*}
    \abs{\bm{u}_{\SSSText{D}}-\bm{u}_{\SSSText{C}}}_{k} & \lesssim_{\rho^\star} \epsilon^2 \RoundBrackets*{\abs{\bm{f}}_{k}+\abs{\tau}_{k+1}},   \\
    \abs{\theta_{\SSSText{D}}-\theta_{\SSSText{C}}}_{k} & \lesssim_{\rho^\star} \epsilon^2 \RoundBrackets*{\abs{\bm{f}}_{k+1}+\abs{\tau}_{k+2}},
  \end{align*}
  where $k > 0$.
\end{corollary}


The remainder of this section aims to clarify that the theoretical results in \cref{thm: main} are, in a sense, optimal. Referring again to \cref{eq: D u,eq: D theta,eq: C u,eq: C theta}, we can observe the following:
\[
  \hat{\bm{u}}_{\SSSText{D}} - \hat{\bm{u}}_{\SSSText{C}} = \RoundBrackets*{B_{\SSSText{D}}^{-1}-B_{\SSSText{C}}^{-1}}\hat{\bm{f}} - \RoundBrackets*{B_{\SSSText{D}}^{-1}\bm{b}_{\SSSText{D}}/c_{\SSSText{D}}-B_{\SSSText{C}}^{-1}\bm{b}_{\SSSText{D}}/c_{\SSSText{D}}}\hat{\tau},
\]
and
\begin{multline*}
  \hat{\theta}_{\SSSText{D}} - \hat{\theta}_{\SSSText{C}} = \RoundBrackets*{B_{\SSSText{D}}^{-1}\bm{b}_{\SSSText{D}}/c_{\SSSText{D}}-B_{\SSSText{C}}^{-1}\bm{b}_{\SSSText{C}}/c_{\SSSText{C}}} \cdot \hat{\bm{f}} \\
  + \RoundBrackets*{{1}/{c_{\SSSText{D}}}-{1}/{c_{\SSSText{C}}}- {\bm{b}_{\SSSText{D}}\cdot B_{\SSSText{D}}^{-1}\bm{b}_{\SSSText{D}}}/{c_{\SSSText{D}}^2}+{\bm{b}_{\SSSText{C}}\cdot B_{\SSSText{C}}^{-1}\bm{b}_{\SSSText{C}}}/{c_{\SSSText{C}}^2}}\hat{\tau}.
\end{multline*}
Therefore, we devise the following indexes:
\begin{gather*}
  \mathbf{Err}_0[i',j']\coloneqq \epsilon^{-2} \norm{B_{\SSSText{D}}^{-1}-B_{\SSSText{C}}^{-1}},\quad \mathbf{Err}_1[i',j']\coloneqq \epsilon^{-2} \frac{\abs{B_{\SSSText{D}}^{-1}\bm{b}_{\SSSText{D}}/c_{\SSSText{D}}-B_{\SSSText{C}}^{-1}\bm{b}_{\SSSText{C}}/c_{\SSSText{C}}}}{\abs{i'}+\abs{j'}}, \\
  \mathbf{Err}_2[i',j']\coloneqq \epsilon^{-2} \frac{\abs{{1}/{c_{\SSSText{D}}}-{1}/{c_{\SSSText{C}}}- {\bm{b}_{\SSSText{D}}\cdot B_{\SSSText{D}}^{-1}\bm{b}_{\SSSText{D}}}/{c_{\SSSText{D}}^2}+{\bm{b}_{\SSSText{C}}\cdot B_{\SSSText{C}}^{-1}\bm{b}_{\SSSText{C}}}/{c_{\SSSText{C}}^2}}}{\abs{i'}^2+\abs{j'}^2}.
\end{gather*}
If the maximal values of $\mathbf{Err}_0[i',j']$, $\mathbf{Err}_1[i',j']$, and $\mathbf{Err}_2[i',j']$ over $(i',j') \in \mathscr{F}_N^\circ$ are all bounded above by $\bigO(1)$ independently of $\epsilon$, then for any $\epsilon$, we can find right-hand terms $(\bm{f}', \tau')$ such that
\begin{align*}
  \norm{\bm{u}_{\SSSText{D}}-\bm{u}_{\SSSText{C}}}_{0} & \gtrsim_{\rho^\star} \epsilon^2 \RoundBrackets*{\norm{\bm{f}'}_{0}+\abs{\tau'}_{1}}, \\
  \norm{\theta_{\SSSText{D}}-\theta_{\SSSText{C}}}_{0} & \gtrsim_{\rho^\star} \epsilon^2 \RoundBrackets*{\abs{\bm{f}'}_{1}+\abs{\tau'}_{2}}.
\end{align*}
Equivalently, the above inequalities imply that the theoretical results in \cref{thm: main} are optimal. We rely on numerical experiments to verify this point.

We select $\epsilon$ from the set $\{\tfrac{1}{17}, \tfrac{1}{33}, \tfrac{1}{65}, \tfrac{1}{129}\}$ and compute the indices $\mathbf{Err}_0[i',j']$, $\mathbf{Err}_1[i',j']$, and $\mathbf{Err}_2[i',j']$ for all $(i',j')\in \mathscr{F}_N$.
To gain insights into the effect of $\rho^\star$, we also vary $\rho^\star$ in $\{0.01, 1.0, 100.0\}$.
Since $\mathbf{Err}_0[i',j']$, $\mathbf{Err}_1[i',j']$, and $\mathbf{Err}_2[i',j']$ are undefined at $(0,0)$, for smooth visualization, we use the average values from the surrounding indices: $(0,1)$, $(1,0)$, $(0,-1)$, and $(-1,0)$.
For instance, $\mathbf{Err}_0[0,0]$ is approximated as $(\mathbf{Err}_0[0,1] + \mathbf{Err}_0[1,0] + \mathbf{Err}_0[0,-1] + \mathbf{Err}_0[-1,0]) / 4$.
The results are shown in \cref{fig: error_u_f,fig: error_u_tau,fig: error_theta_tau}, respectively, for $\mathbf{Err}_0$, $\mathbf{Err}_1$, and $\mathbf{Err}_2$.
In those figures, each row corresponds to a different $\epsilon$, and each column corresponds to a different $\rho^\star$.
Additionally, we provide the minimal and maximal values in a format $(\mathup{min},\mathup{max})$ of the index for the specific setting above the sub-figure.
From the figures, we can observe that by scaling $(i',j')$ to $(i'/N, j'/N)\in D$, all indexes exhibit a convergence behavior as $\epsilon$ decreases.
Specifically, we can see that the maximal and minimal values of $\mathbf{Err}_0[i',j']$, $\mathbf{Err}_1[i',j']$, and $\mathbf{Err}_2[i',j']$ over $(i',j') \in \mathscr{F}_N^\circ$ approach fixed values.
Therefore, we can confirm that the theoretical results in \cref{thm: main} are optimal.
As expected, for different $\rho^\star$, the limiting pattern of $\mathbf{Err}_0$, $\mathbf{Err}_1$, and $\mathbf{Err}_2$ are different.
It seems that $\mathbf{Err}_0$ and $\mathbf{Err}_1$ are more vulnerable to $\rho^\star$ (decreasing by a factor of $10$ as $\rho^\star$ increases to $10^2$), while $\mathbf{Err}_2$ shows less sensitivity to $\rho^\star$.

\begin{figure}[!ht]
  \centering
  \resizebox{\textwidth}{!}{\input{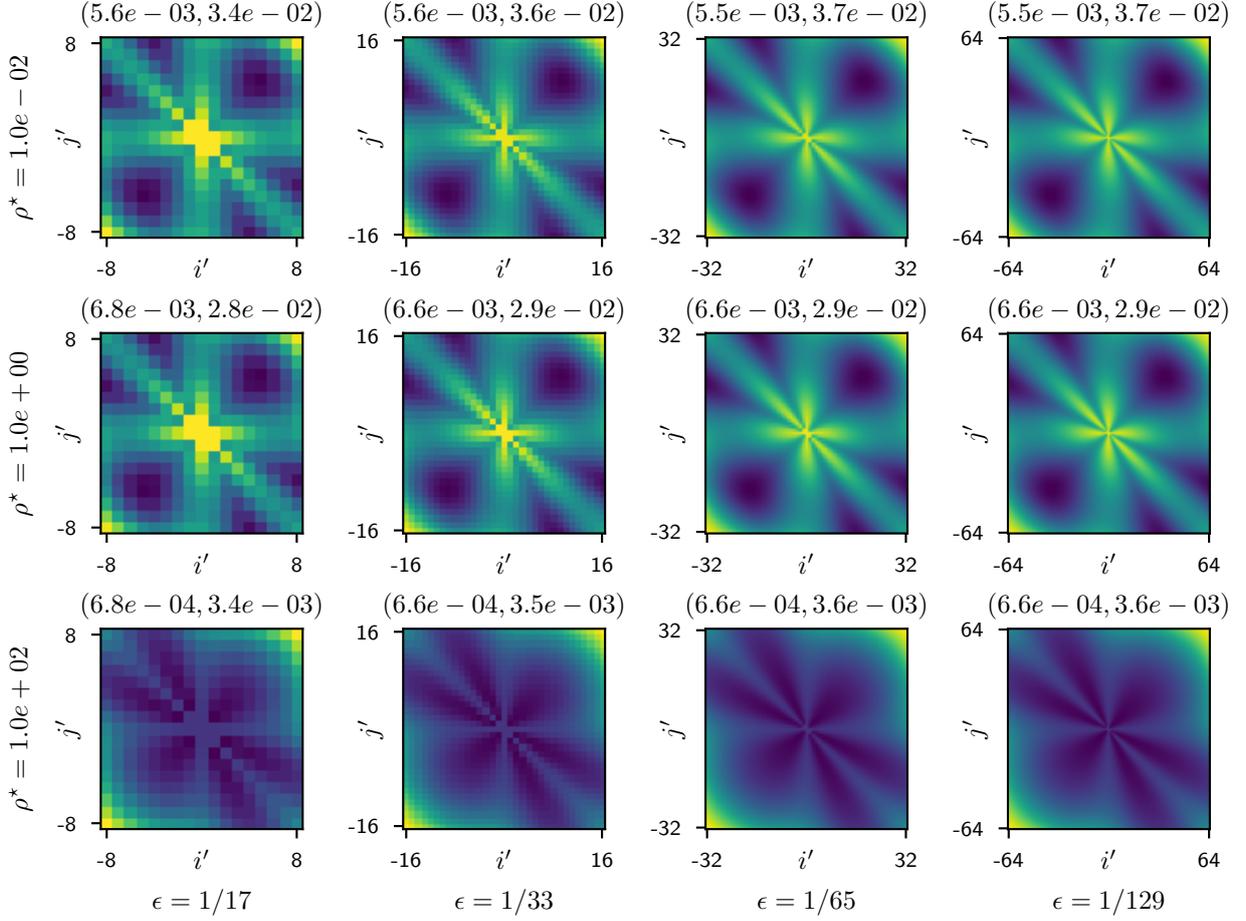}}
  \caption{The plot of $\mathbf{Err}_0[i',j']$ with $(i',j')\in \mathscr{F}_N$ under different settings is displayed.
  Each row corresponds to a different $\epsilon$, and each column corresponds to a different $\rho^\star$.
  The minimal and maximal values of each setting are indicated above the sub-figure.}\label{fig: error_u_f}
\end{figure}

\begin{figure}[!ht]
  \centering
  \resizebox{\textwidth}{!}{\input{figs/error_u_tau.pgf}}
  \caption{The plot of $\mathbf{Err}_1[i',j']$ with $(i',j')\in \mathscr{F}_N$ under different settings is displayed.
  Each row corresponds to a different $\epsilon$, and each column corresponds to a different $\rho^\star$.
  The minimal and maximal values of each setting are indicated above the sub-figure.}\label{fig: error_u_tau}
\end{figure}

\begin{figure}[!ht]
  \centering
  \resizebox{\textwidth}{!}{\input{figs/error_theta_tau.pgf}}
  \caption{The plot of $\mathbf{Err}_2[i',j']$ with $(i',j')\in \mathscr{F}_N$ under different settings is displayed.
  Each row corresponds to a different $\epsilon$, and each column corresponds to a different $\rho^\star$.
  The minimal and maximal values of each setting are indicated above the sub-figure.}\label{fig: error_theta_tau}
\end{figure}

\section{Extension on the rectangular lattice}\label{sec: extensions}

The homogenization of the rectangular lattice parallels that of the triangular lattice in \cref{fig: triangular lattice}.
In this setting, two beam types arise, with directions
\(\bm{l}_1=[1,0]^\intercal\) and \(\bm{l}_2=[0,1]^\intercal\), and translation vectors
\(\bm{t}_x=[1,0]^\intercal\) and \(\bm{t}_y=[0,1]^\intercal\).
The analogue of \cref{thm: main} for the rectangular lattice can be stated accordingly; detailed derivations are omitted for brevity.
This section concentrates on the mechanical interpretation of the resulting homogenized models.

Because the lattice axes now coincide with the Cartesian axes, the operators $\partial_\alpha$ and $\partial_\beta$ reduce directly to $\partial_x$ and $\partial_y$, respectively.
Following the derivation in \cref{sec: homogenization}, the homogenized (continuous) model reads
\begin{subequations}
  \begin{align}
    -\nabla \cdot
    \begin{bmatrix}
      \rho^\star\,\partial_x u^x                    & 12\,\partial_x u^y - 12\,\theta_{\SSSText{C}} \\
      12\,\partial_y u^x + 12\,\theta_{\SSSText{C}} & \rho^\star\,\partial_y u^y
    \end{bmatrix}
     & = \bm{f}_{\SSSText{C}}, \label{eq: linear-mum-balance} \\
    -\bigl(12\,\partial_x u^y - 12\,\partial_y u^x - 24\,\theta_{\SSSText{C}}\bigr)
     & = \tau_{\SSSText{C}}, \label{eq: ang-mum-balance}
  \end{align}
\end{subequations}
where $\bm{u}_{\SSSText{C}}=\SquareBrackets{u^x,u^y}^\intercal$ and $\theta_{\SSSText{C}}$ denote the continuum displacement and microrotation.
Defining the (generally non-symmetric) stress tensor
\[
  \sigma =
  \begin{bmatrix}
    \rho^\star\,\partial_x u^x      & 12\,\partial_x u^y - 12\,\theta \\
    12\,\partial_y u^x + 12\,\theta & \rho^\star\,\partial_y u^y
  \end{bmatrix},
\]
\cref{eq: linear-mum-balance} enforces balance of linear momentum, while \cref{eq: ang-mum-balance} enforces balance of angular momentum.
The inclusion of the rotation variable breaks the symmetry of $\sigma$, in contrast to classical Cauchy elasticity.

\section{Conclusion}\label{sec: conclusion}
The paper addresses the homogenization problem of beam lattices, and
using the triangular lattice as a case study, we devise a mechanical problem with PBCs that its convergence to the continuum model is well-perceived.
We demonstrate that through the low-frequency assumption on external forces and torques, the PDE model can be convincingly determined.
However, to quantify the homogenization error, we observe numerically that the correct theory is not as straightforward as the periodic homogenization of elliptic PDEs.
Essentially, the rotation field renders the lattice system less coercive and more challenging.
To overcome this difficulty, we introduce a new identity that decouples the Schur complement into non-negative quadratic forms.
This allows us to establish the coercivity of the Schur complement, which is crucial for quantitative homogenization.
We confirm the optimality of all rates predicted by our theory through numerical experiments.
We extend this strategy to rectangular and parallelogram lattices, achieving similar results.

We believe several aspects of our work can be further developed to address broader challenges and refine existing methodologies.
First, our current results are limited to PBCs, which allow for the effective application of Fourier analysis.
However, other boundary conditions, such as Dirichlet boundary conditions, are equally important in practical scenarios.
We anticipate that our techniques can be extended to accommodate more general boundary conditions using generalized trigonometric transformations, such as sine transformations.
Second, as demonstrated in this study, when the homogenized model is independent of the scale parameter, further improvement in the homogenization error becomes unattainable.
While higher-order models can be derived using simple Taylor expansions, this approach often compromises coercivity, limiting the rigor of quantitative analysis like that presented in this work.
To overcome this limitation, we propose exploring the derivation of higher-order models through Fourier analysis while simultaneously achieving rigorous quantitative results.

%% file: figs/EB-beam.pgf
\begingroup%
\makeatletter%
\begin{pgfpicture}%
\pgfpathrectangle{\pgfpointorigin}{\pgfqpoint{3.743927in}{1.741660in}}%
\pgfusepath{use as bounding box, clip}%
\begin{pgfscope}%
\pgfsetbuttcap%
\pgfsetmiterjoin%
\definecolor{currentfill}{rgb}{1.000000,1.000000,1.000000}%
\pgfsetfillcolor{currentfill}%
\pgfsetlinewidth{0.000000pt}%
\definecolor{currentstroke}{rgb}{1.000000,1.000000,1.000000}%
\pgfsetstrokecolor{currentstroke}%
\pgfsetdash{}{0pt}%
\pgfpathmoveto{\pgfqpoint{0.000000in}{0.000000in}}%
\pgfpathlineto{\pgfqpoint{3.743927in}{0.000000in}}%
\pgfpathlineto{\pgfqpoint{3.743927in}{1.741660in}}%
\pgfpathlineto{\pgfqpoint{0.000000in}{1.741660in}}%
\pgfpathlineto{\pgfqpoint{0.000000in}{0.000000in}}%
\pgfpathclose%
\pgfusepath{fill}%
\end{pgfscope}%
\begin{pgfscope}%
\pgfpathrectangle{\pgfqpoint{0.275483in}{0.100000in}}{\pgfqpoint{2.929497in}{1.541660in}}%
\pgfusepath{clip}%
\pgfsetbuttcap%
\pgfsetmiterjoin%
\pgfsetlinewidth{2.007500pt}%
\definecolor{currentstroke}{rgb}{0.000000,0.501961,0.000000}%
\pgfsetstrokecolor{currentstroke}%
\pgfsetdash{}{0pt}%
\pgfpathmoveto{\pgfqpoint{2.197946in}{0.841064in}}%
\pgfpathcurveto{\pgfqpoint{2.436520in}{1.254287in}}{\pgfqpoint{2.675094in}{1.049817in}}{\pgfqpoint{2.913668in}{1.288390in}}%
\pgfusepath{stroke}%
\end{pgfscope}%
\begin{pgfscope}%
\pgfpathrectangle{\pgfqpoint{0.275483in}{0.100000in}}{\pgfqpoint{2.929497in}{1.541660in}}%
\pgfusepath{clip}%
\pgfsetroundcap%
\pgfsetroundjoin%
\definecolor{currentfill}{rgb}{0.501961,0.501961,0.501961}%
\pgfsetfillcolor{currentfill}%
\pgfsetfillopacity{0.500000}%
\pgfsetlinewidth{1.003750pt}%
\definecolor{currentstroke}{rgb}{0.501961,0.501961,0.501961}%
\pgfsetstrokecolor{currentstroke}%
\pgfsetstrokeopacity{0.500000}%
\pgfsetdash{}{0pt}%
\pgfpathmoveto{\pgfqpoint{2.228460in}{0.656210in}}%
\pgfpathlineto{\pgfqpoint{2.225895in}{0.662120in}}%
\pgfpathquadraticcurveto{\pgfqpoint{2.284878in}{0.651561in}}{\pgfqpoint{2.319829in}{0.688384in}}%
\pgfpathlineto{\pgfqpoint{2.303779in}{0.689548in}}%
\pgfpathquadraticcurveto{\pgfqpoint{2.322851in}{0.700344in}}{\pgfqpoint{2.344079in}{0.725272in}}%
\pgfpathquadraticcurveto{\pgfqpoint{2.341233in}{0.692069in}}{\pgfqpoint{2.322306in}{0.668851in}}%
\pgfpathlineto{\pgfqpoint{2.322321in}{0.686170in}}%
\pgfpathquadraticcurveto{\pgfqpoint{2.288418in}{0.645476in}}{\pgfqpoint{2.224445in}{0.651104in}}%
\pgfpathlineto{\pgfqpoint{2.228460in}{0.656210in}}%
\pgfpathlineto{\pgfqpoint{2.228460in}{0.656210in}}%
\pgfpathclose%
\pgfusepath{stroke,fill}%
\end{pgfscope}%
\begin{pgfscope}%
\pgfpathrectangle{\pgfqpoint{0.275483in}{0.100000in}}{\pgfqpoint{2.929497in}{1.541660in}}%
\pgfusepath{clip}%
\pgfsetroundcap%
\pgfsetroundjoin%
\definecolor{currentfill}{rgb}{0.501961,0.501961,0.501961}%
\pgfsetfillcolor{currentfill}%
\pgfsetfillopacity{0.500000}%
\pgfsetlinewidth{1.003750pt}%
\definecolor{currentstroke}{rgb}{0.501961,0.501961,0.501961}%
\pgfsetstrokecolor{currentstroke}%
\pgfsetstrokeopacity{0.500000}%
\pgfsetdash{}{0pt}%
\pgfpathmoveto{\pgfqpoint{2.883894in}{1.476429in}}%
\pgfpathlineto{\pgfqpoint{2.886023in}{1.470338in}}%
\pgfpathquadraticcurveto{\pgfqpoint{2.856555in}{1.479345in}}{\pgfqpoint{2.832827in}{1.468832in}}%
\pgfpathlineto{\pgfqpoint{2.847081in}{1.461541in}}%
\pgfpathquadraticcurveto{\pgfqpoint{2.825897in}{1.460460in}}{\pgfqpoint{2.799885in}{1.439597in}}%
\pgfpathquadraticcurveto{\pgfqpoint{2.808975in}{1.473017in}}{\pgfqpoint{2.838711in}{1.488028in}}%
\pgfpathlineto{\pgfqpoint{2.831297in}{1.471794in}}%
\pgfpathquadraticcurveto{\pgfqpoint{2.854841in}{1.485701in}}{\pgfqpoint{2.888353in}{1.481202in}}%
\pgfpathlineto{\pgfqpoint{2.883894in}{1.476429in}}%
\pgfpathlineto{\pgfqpoint{2.883894in}{1.476429in}}%
\pgfpathclose%
\pgfusepath{stroke,fill}%
\end{pgfscope}%
\begin{pgfscope}%
\pgfpathrectangle{\pgfqpoint{0.275483in}{0.100000in}}{\pgfqpoint{2.929497in}{1.541660in}}%
\pgfusepath{clip}%
\pgfsetrectcap%
\pgfsetroundjoin%
\pgfsetlinewidth{2.007500pt}%
\definecolor{currentstroke}{rgb}{0.000000,0.501961,0.000000}%
\pgfsetstrokecolor{currentstroke}%
\pgfsetdash{}{0pt}%
\pgfpathmoveto{\pgfqpoint{0.408642in}{0.393738in}}%
\pgfpathlineto{\pgfqpoint{1.303294in}{0.393738in}}%
\pgfusepath{stroke}%
\end{pgfscope}%
\begin{pgfscope}%
\pgfpathrectangle{\pgfqpoint{0.275483in}{0.100000in}}{\pgfqpoint{2.929497in}{1.541660in}}%
\pgfusepath{clip}%
\pgfsetbuttcap%
\pgfsetroundjoin%
\pgfsetlinewidth{1.003750pt}%
\definecolor{currentstroke}{rgb}{0.501961,0.501961,0.501961}%
\pgfsetstrokecolor{currentstroke}%
\pgfsetstrokeopacity{0.500000}%
\pgfsetdash{{3.700000pt}{1.600000pt}}{0.000000pt}%
\pgfpathmoveto{\pgfqpoint{0.408642in}{0.170075in}}%
\pgfpathlineto{\pgfqpoint{0.408642in}{0.617401in}}%
\pgfusepath{stroke}%
\end{pgfscope}%
\begin{pgfscope}%
\pgfpathrectangle{\pgfqpoint{0.275483in}{0.100000in}}{\pgfqpoint{2.929497in}{1.541660in}}%
\pgfusepath{clip}%
\pgfsetbuttcap%
\pgfsetroundjoin%
\pgfsetlinewidth{1.003750pt}%
\definecolor{currentstroke}{rgb}{0.501961,0.501961,0.501961}%
\pgfsetstrokecolor{currentstroke}%
\pgfsetstrokeopacity{0.500000}%
\pgfsetdash{{3.700000pt}{1.600000pt}}{0.000000pt}%
\pgfpathmoveto{\pgfqpoint{1.303294in}{0.170075in}}%
\pgfpathlineto{\pgfqpoint{1.303294in}{0.617401in}}%
\pgfusepath{stroke}%
\end{pgfscope}%
\begin{pgfscope}%
\pgfpathrectangle{\pgfqpoint{0.275483in}{0.100000in}}{\pgfqpoint{2.929497in}{1.541660in}}%
\pgfusepath{clip}%
\pgfsetbuttcap%
\pgfsetroundjoin%
\pgfsetlinewidth{1.003750pt}%
\definecolor{currentstroke}{rgb}{0.501961,0.501961,0.501961}%
\pgfsetstrokecolor{currentstroke}%
\pgfsetstrokeopacity{0.500000}%
\pgfsetdash{{3.700000pt}{1.600000pt}}{0.000000pt}%
\pgfpathmoveto{\pgfqpoint{2.391644in}{0.729233in}}%
\pgfpathlineto{\pgfqpoint{2.004248in}{0.952896in}}%
\pgfusepath{stroke}%
\end{pgfscope}%
\begin{pgfscope}%
\pgfpathrectangle{\pgfqpoint{0.275483in}{0.100000in}}{\pgfqpoint{2.929497in}{1.541660in}}%
\pgfusepath{clip}%
\pgfsetbuttcap%
\pgfsetroundjoin%
\pgfsetlinewidth{1.003750pt}%
\definecolor{currentstroke}{rgb}{0.501961,0.501961,0.501961}%
\pgfsetstrokecolor{currentstroke}%
\pgfsetstrokeopacity{0.500000}%
\pgfsetdash{{3.700000pt}{1.600000pt}}{0.000000pt}%
\pgfpathmoveto{\pgfqpoint{2.197946in}{0.617401in}}%
\pgfpathlineto{\pgfqpoint{2.197946in}{1.064727in}}%
\pgfusepath{stroke}%
\end{pgfscope}%
\begin{pgfscope}%
\pgfpathrectangle{\pgfqpoint{0.275483in}{0.100000in}}{\pgfqpoint{2.929497in}{1.541660in}}%
\pgfusepath{clip}%
\pgfsetbuttcap%
\pgfsetroundjoin%
\pgfsetlinewidth{1.003750pt}%
\definecolor{currentstroke}{rgb}{0.501961,0.501961,0.501961}%
\pgfsetstrokecolor{currentstroke}%
\pgfsetstrokeopacity{0.500000}%
\pgfsetdash{{3.700000pt}{1.600000pt}}{0.000000pt}%
\pgfpathmoveto{\pgfqpoint{3.071821in}{1.130237in}}%
\pgfpathlineto{\pgfqpoint{2.755514in}{1.446544in}}%
\pgfusepath{stroke}%
\end{pgfscope}%
\begin{pgfscope}%
\pgfpathrectangle{\pgfqpoint{0.275483in}{0.100000in}}{\pgfqpoint{2.929497in}{1.541660in}}%
\pgfusepath{clip}%
\pgfsetbuttcap%
\pgfsetroundjoin%
\pgfsetlinewidth{1.003750pt}%
\definecolor{currentstroke}{rgb}{0.501961,0.501961,0.501961}%
\pgfsetstrokecolor{currentstroke}%
\pgfsetstrokeopacity{0.500000}%
\pgfsetdash{{3.700000pt}{1.600000pt}}{0.000000pt}%
\pgfpathmoveto{\pgfqpoint{2.913668in}{1.064727in}}%
\pgfpathlineto{\pgfqpoint{2.913668in}{1.512053in}}%
\pgfusepath{stroke}%
\end{pgfscope}%
\begin{pgfscope}%
\definecolor{textcolor}{rgb}{0.000000,0.000000,0.000000}%
\pgfsetstrokecolor{textcolor}%
\pgfsetfillcolor{textcolor}%
\pgftext[x=0.408642in,y=0.393738in,right,]{\color{textcolor}{\sffamily\fontsize{10.000000}{12.000000}\selectfont\catcode`\^=\active\def^{\ifmmode\sp\else\^{}\fi}\catcode`\%=\active\def
\end{pgfscope}%
\begin{pgfscope}%
\definecolor{textcolor}{rgb}{0.000000,0.000000,0.000000}%
\pgfsetstrokecolor{textcolor}%
\pgfsetfillcolor{textcolor}%
\pgftext[x=1.303294in,y=0.393738in,left,]{\color{textcolor}{\sffamily\fontsize{10.000000}{12.000000}\selectfont\catcode`\^=\active\def^{\ifmmode\sp\else\^{}\fi}\catcode`\%=\active\def
\end{pgfscope}%
\begin{pgfscope}%
\definecolor{textcolor}{rgb}{0.000000,0.000000,0.000000}%
\pgfsetstrokecolor{textcolor}%
\pgfsetfillcolor{textcolor}%
\pgftext[x=2.197946in,y=0.841064in,right,top]{\color{textcolor}{\sffamily\fontsize{10.000000}{12.000000}\selectfont\catcode`\^=\active\def^{\ifmmode\sp\else\^{}\fi}\catcode`\%=\active\def
\end{pgfscope}%
\begin{pgfscope}%
\definecolor{textcolor}{rgb}{0.000000,0.000000,0.000000}%
\pgfsetstrokecolor{textcolor}%
\pgfsetfillcolor{textcolor}%
\pgftext[x=2.376877in,y=0.662134in,,top]{\color{textcolor}{\sffamily\fontsize{10.000000}{12.000000}\selectfont\catcode`\^=\active\def^{\ifmmode\sp\else\^{}\fi}\catcode`\%=\active\def
\end{pgfscope}%
\begin{pgfscope}%
\definecolor{textcolor}{rgb}{0.000000,0.000000,0.000000}%
\pgfsetstrokecolor{textcolor}%
\pgfsetfillcolor{textcolor}%
\pgftext[x=2.913668in,y=1.288390in,left,bottom]{\color{textcolor}{\sffamily\fontsize{10.000000}{12.000000}\selectfont\catcode`\^=\active\def^{\ifmmode\sp\else\^{}\fi}\catcode`\%=\active\def
\end{pgfscope}%
\begin{pgfscope}%
\definecolor{textcolor}{rgb}{0.000000,0.000000,0.000000}%
\pgfsetstrokecolor{textcolor}%
\pgfsetfillcolor{textcolor}%
\pgftext[x=2.824203in,y=1.512053in,,bottom]{\color{textcolor}{\sffamily\fontsize{10.000000}{12.000000}\selectfont\catcode`\^=\active\def^{\ifmmode\sp\else\^{}\fi}\catcode`\%=\active\def
\end{pgfscope}%
\begin{pgfscope}%
\pgfsetbuttcap%
\pgfsetmiterjoin%
\definecolor{currentfill}{rgb}{1.000000,1.000000,0.000000}%
\pgfsetfillcolor{currentfill}%
\pgfsetlinewidth{0.240900pt}%
\definecolor{currentstroke}{rgb}{1.000000,1.000000,0.000000}%
\pgfsetstrokecolor{currentstroke}%
\pgfsetdash{}{0pt}%
\pgfpathmoveto{\pgfqpoint{0.246206in}{0.723781in}}%
\pgfpathlineto{\pgfqpoint{1.465730in}{0.723781in}}%
\pgfpathlineto{\pgfqpoint{1.465730in}{0.958348in}}%
\pgfpathlineto{\pgfqpoint{0.246206in}{0.958348in}}%
\pgfpathlineto{\pgfqpoint{0.246206in}{0.723781in}}%
\pgfpathclose%
\pgfusepath{stroke,fill}%
\end{pgfscope}%
\begin{pgfscope}%
\definecolor{textcolor}{rgb}{0.000000,0.000000,0.000000}%
\pgfsetstrokecolor{textcolor}%
\pgfsetfillcolor{textcolor}%
\pgftext[x=0.855968in,y=0.841064in,,]{\color{textcolor}{\sffamily\fontsize{10.000000}{12.000000}\selectfont\catcode`\^=\active\def^{\ifmmode\sp\else\^{}\fi}\catcode`\%=\active\def
\end{pgfscope}%
\begin{pgfscope}%
\pgfsetbuttcap%
\pgfsetmiterjoin%
\definecolor{currentfill}{rgb}{1.000000,1.000000,0.000000}%
\pgfsetfillcolor{currentfill}%
\pgfsetlinewidth{0.240900pt}%
\definecolor{currentstroke}{rgb}{1.000000,1.000000,0.000000}%
\pgfsetstrokecolor{currentstroke}%
\pgfsetdash{}{0pt}%
\pgfpathmoveto{\pgfqpoint{1.984625in}{0.276455in}}%
\pgfpathlineto{\pgfqpoint{3.126989in}{0.276455in}}%
\pgfpathlineto{\pgfqpoint{3.126989in}{0.511022in}}%
\pgfpathlineto{\pgfqpoint{1.984625in}{0.511022in}}%
\pgfpathlineto{\pgfqpoint{1.984625in}{0.276455in}}%
\pgfpathclose%
\pgfusepath{stroke,fill}%
\end{pgfscope}%
\begin{pgfscope}%
\definecolor{textcolor}{rgb}{0.000000,0.000000,0.000000}%
\pgfsetstrokecolor{textcolor}%
\pgfsetfillcolor{textcolor}%
\pgftext[x=2.555807in,y=0.393738in,,]{\color{textcolor}{\sffamily\fontsize{10.000000}{12.000000}\selectfont\catcode`\^=\active\def^{\ifmmode\sp\else\^{}\fi}\catcode`\%=\active\def
\end{pgfscope}%
\end{pgfpicture}%
\makeatother%
\endgroup%

%% file: acknowledgements.tex
EC's research is partially supported by the Hong Kong RGC General Research Fund (Project numbers: 14305222, 14304021, and 14305624).
CY's research is partially supported by the Xidian University Specially Funded Project for Interdisciplinary Exploration (No. TZJH2024008).